\documentclass[a4paper,10pt,reqno]{amsart}
\usepackage{ae}
\usepackage{aeguill}
\usepackage[utf8]{inputenc}
\usepackage[T1]{fontenc}
\usepackage{amssymb,amsmath,amsthm,mathrsfs,graphics,graphicx,float,lscape}
\usepackage[english,frenchb]{babel}
\usepackage{graphicx}
\usepackage{latexsym}
\usepackage{CJK}
\usepackage[all]{xy}
\usepackage{lscape}
\usepackage{cancel}
\usepackage{color}
\usepackage{here}
\allowdisplaybreaks[1]

\usepackage{fancyhdr}
\usepackage{hyperref}

\newcommand{\comm}[1]{}
\newcommand{\comment}[1]{}

\newcommand{\ntop}[2]{\genfrac{}{}{0pt}{1}{#1}{#2}}

\def\R{\mathbb{R}}
\def\C{\mathbb{C}}
\def\Q{\mathbb{Q}}
\def\Z{\mathbb{Z}}
\def\N{\mathbb{N}}

\def\restriction#1#2{\mathchoice
              {\setbox1\hbox{${\displaystyle #1}_{\scriptstyle #2}$}
              \restrictionaux{#1}{#2}}
              {\setbox1\hbox{${\textstyle #1}_{\scriptstyle #2}$}
              \restrictionaux{#1}{#2}}
              {\setbox1\hbox{${\scriptstyle #1}_{\scriptscriptstyle #2}$}
              \restrictionaux{#1}{#2}}
              {\setbox1\hbox{${\scriptscriptstyle #1}_{\scriptscriptstyle #2}$}
              \restrictionaux{#1}{#2}}}
\def\restrictionaux#1#2{{#1\,\smash{\vrule height .8\ht1 depth .85\dp1}}_{\,#2}}

\newtheorem{theo}{\textsc{Theorem}}[section]
\newtheorem{thm}[theo]{\textsc{Theorem}}
\newtheorem{prop}[theo]{\textsc{Proposition}}

\newtheorem{lem}[theo]{\textsc{Lemma}}
\newtheorem{lemma}[theo]{\textsc{Lemma}}

\newtheorem{rem}[theo]{\textsc{Remark}}
\newtheorem{defi}[theo]{\textsc{Definition}}

\newtheorem{theoalph}{Theorem}

\newtheorem*{theorem}{\textsc{Theorem}}

\makeatletter
\renewcommand\subsection{\@startsection{subsection}{2}%
  \z@{-.5\linespacing\@plus-.7\linespacing}{.5\linespacing}%
  {\normalfont\bfseries}}
\renewcommand\subsubsection{\@startsection{subsubsection}{3}%
  \z@{-.5\linespacing\@plus-.7\linespacing}{.5\linespacing}%
  {\normalfont\bfseries\itshape}}
\makeatother

\begin{document}

\selectlanguage{english}

\title[Weak mixing properties of i.e.t.'s \& translation flows]{Weak mixing properties of interval exchange transformations \& translation flows}     

\author{Artur \textsc{Avila} \& Martin \textsc{Leguil}}  

\thanks{A.A. was partially supported by the ERC Starting Grant Quasiperiodic.}

\date{}


\address{Institut de Mathématiques de Jussieu - Paris Rive Gauche, CNRS UMR 7586,
Université Paris Diderot, Sorbonne Paris Cité
Sorbonnes Universités, UPMC Université Paris 06, F-75013, Paris, France \& IMPA,
Estrada Dona Castorina 110, 22460-320, Rio de Janeiro, Brazil.}

\email{artur@math.univ-paris-diderot.fr}

\address{Universit\'e Paris Diderot, Sorbonne Paris Cit\'e, Institut de Math\'ematiques de
Jussieu-Paris Rive Gauche, UMR $7586$, CNRS, Sorbonne Universit\'es, UPMC Université Paris $06$,
F-$75013$ Paris, France.}

\email{martin.leguil@imj-prg.fr}

\maketitle

\begin{abstract}
\noindent Let $d >1$. In this paper we show that for an irreducible permutation  $\pi$ which is not a rotation, the set of $[\lambda]\in \mathbb{P}_+^{d-1}$ such that the interval exchange transformation $f([\lambda],\pi)$ is not weakly mixing does not have full Hausdorff dimension. We also obtain an analogous statement for translation flows. In particular, it strengthens the result of almost sure weak mixing proved by G. Forni and the first author in \cite{AF1}. We adapt here the probabilistic argument developed in their paper in order to get some large deviation results. We then show how the latter can be converted into estimates on the Hausdorff dimension of the set of ``bad" parameters in the context of \textit{fast decaying} cocycles, following the strategy of \cite{AD}. 
\end{abstract}

\tableofcontents


\section*{Introduction}

An \textit{interval exchange transformation}, or i.e.t., is a piecewise order-preserving bijection $f$ of an interval $I$ on the real axis. More precisely, $I$ splits into a finite number of subintervals $(I_i)_{i=1,\dots,d}$, $d>1$, such that the restriction of $f$ to each of them is a translation. The map $f$ is completely described by a pair $(\lambda, \pi)\in \R_+^d \times \mathfrak{S}_d$: $\lambda$ is a vector whose coordinates $\lambda_i:=|I_i|$ correspond to the lengths of the subintervals, and $\pi$ a combinatorial data which  prescribes in which way the different subintervals are reordered after application of $f$. We will write $f=f(\lambda,\pi)$. In the following, we will mostly consider \textit{irreducible} permutations $\pi$, which we denote by $\pi \in \mathfrak{S}_d^0$; this somehow expresses that the dynamics is ``indecomposable''. Since dilations on $\lambda$ do not change the dynamics of the i.e.t., we will also sometimes use the notation $f([\lambda],\pi)$, with $[\lambda] \in \mathbb{P}_+^{d-1}$. For more details on interval exchange transformations we refer to Section \ref{sectionintervalexchangerenormalization}.

A \textit{translation surface} is a pair $(S,\omega)$ where $S$ is a surface and $\omega$ some nonzero Abelian differential defined on it. Denote by $\Sigma \subset S$ the set of zeros of $\omega$; its complement $S\backslash \Sigma$ admits an atlas such that transition maps between two charts are just translations (see Subsection \ref{translsur} for more details on translation surfaces). Interval exchange transformations can be seen as a discrete version of the geodesic flow on some translation surface, also called a \textit{translation flow}. The introduction of these objects was motivated by the study of the billiard flow on rational polygons, i.e., whose angles are commensurate to $\pi$; the relation between these problems is given by a construction called \textit{unfolding}, which associates a translation surface to such a polygon (see for instance \cite{Z2}) and makes the billiard flow into some translation flow.

In this paper, we are interested in the ergodic properties of interval exchange transformations on $d>1$ subintervals. It is clear that such transformations preserve the Lebesgue measure. In fact this is often the unique invariant measure: Masur in \cite{Ma1} and Veech in \cite{V2} have shown that if the permutation $\pi$ is irreducible, then for Lebesgue-almost every $[\lambda] \in \mathbb{P}^{d-1}_+$, the i.e.t.\ $f([\lambda],\pi)$ is uniquely ergodic. As a by-product of our methods, we will see here that in fact, the set of $[\lambda]$ such that $f([\lambda],\pi)$ is not uniquely ergodic does not have full Hausdorff dimension.\footnote{It was pointed out to us by J. Athreya and J. Chaika that more is true actually: the results of Masur \cite{Ma2} imply that the Hausdorff codimension of $[\lambda] \in \mathbb{P}^{d-1}_+$ such that $f([\lambda],\pi)$ is not uniquely ergodic is at least $1/2$.}

In another direction, Katok has proved that i.e.t.'s and suspension flows over i.e.t.'s with roof function of bounded variation are never mixing with respect to Lebesgue measure, see \cite{Ka}. Basically, what mixing expresses is that the position of a point at time $n$ is almost independent of its initial position when $n \geq 0$ is large. Let us recall that a measure-preserving transformation $f$ of a probability space $(X,m)$ is said to be \textit{weakly mixing} if for every pair of measurable sets $A, B \in X$, there exists a subset $J(A,B) \subset \N$ of density zero such that 
\begin{equation}\label{weakmm}
\lim\limits_{J(A,B)\not\ni n \to +\infty} m(f^{-n}(A) \cap B) = m(A) m(B).
\end{equation}
It follows from this definition that every mixing transformation is weakly mixing, and every weakly mixing transformation is ergodic.\footnote{Indeed, mixing holds when we can take $J(A,B)=\emptyset$ in (\ref{weakmm}). For ergodicity, if the set $A$ is $f$-invariant, the choice $B=A$ in (\ref{weakmm}) yields that $A$ has either full or zero measure.} 

From the previous discussion, it is therefore natural to ask whether a typical i.e.t.\ is weakly mixing or not; this point is more delicate except in the case where the permutation $\pi$ associated to the i.e.t.\ $f(\lambda,\pi)$ is a \textit{rotation} of $\{1,\dots,d\}$, i.e., $\pi(i+1)\equiv \pi(i) +1 \mod d$, for all $i \in \{1,\dots,d\}$. Indeed, in this case, the i.e.t.\ $f(\lambda,\pi)$ is conjugate to a rotation of the circle, hence it is not weakly mixing, for every $\lambda \in \R_+^d$.

It is a classical fact that any invertible measure-preserving transformation $f$ is weakly mixing if and only if it has \textit{continuous spectrum}, that is, the only eigenvalue of $f$ is $1$ and the only eigenfunctions are constants. To prove weak mixing we thus rule out the existence of non-constant measurable eigenfunctions.

Let us recall some previous advances in the problem of the prevalence of weak mixing among interval exchange transformations. Partial results in this direction had been obtained by Katok and Stepin \cite{KS1}, who proved weak mixing for almost all i.e.t.'s on 3 intervals. In \cite{V4}, Veech has shown that weak mixing holds for infinitely many irreducible permutations.

The question of almost sure weak mixing for i.e.t.'s was first fully answered by Forni and the first author in \cite{AF1}, where the following result is proved:

\begin{theorem}[Theorem A, Avila-Forni \cite{AF1}]
Let $\pi$ be an irreducible permutation of $\{1,\dots,d\}$ which is not a rotation. For Lebesgue almost every $\lambda \in \R_+^d$, the i.e.t.\ $f(\lambda,\pi)$ is weakly mixing.
\end{theorem}

They also obtain an analogous statement for translation flows.
\begin{theorem}[Theorem B, Avila-Forni \cite{AF1}]
For almost every translation surface $(S,\omega)$ in a given stratum of the moduli space of translation surfaces of genus $g > 1$, the translation flow on $(S,\omega)$ is weakly mixing in almost every direction.
\end{theorem}
Although translation flows can be seen as suspension flows over i.e.t.'s, the second result is not a direct consequence of the first one since the property of weak mixing is not invariant under suspensions and time changes.\\

Note that the previous results tell nothing about zero measure subsets of the moduli space of translation surfaces; in particular, the question of weak mixing was still open for translation flows on Veech surfaces, which are exceptionally symmetric translation surfaces associated to the dynamics of rational polygonal billiards. This problem was solved in \cite{AD} by Delecroix and the first author:

\begin{theorem}[Theorem 2, Avila-Delecroix \cite{AD}]
The geodesic flow in a non-arithmetic Veech surface is weakly mixing in almost every direction. Indeed, the set of exceptional directions has Hausdorff dimension less than one.
\end{theorem}

The proof of almost sure weak mixing in \cite{AF1} is based on some parameter exclusion; however this reasoning is not adapted to the particular case of translation flows on Veech surfaces. In \cite{AD}, the authors have developed another strategy to deal with the problem of prevalent weak mixing, based on considerations on Hausdorff dimension and its link to some property that we refer to as \textit{fast decay} in what follows (see Subsection \ref{fastdecaydefin} for the definition).\\

In the present paper, we improve the ``almost sure" statement obtained in \cite{AF1}; the proof we give owes much to the ideas developed in \cite{AF1} and \cite{AD}. Our main result is the following.

\begin{theoalph}
Let $d>1$ and let $\pi \in \mathfrak{S}_d^0$ be an irreducible permutation which is not a rotation; then the set of $[\lambda] \in \mathbb{P}^{d-1}_+$ such that $f([\lambda],\pi)$ is not weakly mixing has Hausdorff dimension strictly less than $d-1$.
\end{theoalph}

We also get a similar statement for translation flows. Let $d>1$, and $\pi \in \mathfrak{S}_d^0$ which is not a \text{rotation;} we consider the translation flows which are parametrized by a pair $(h,[\lambda]) \in H(\pi) \times \mathbb{P}_+^{d-1}$, where $\mathrm{dim}(H(\pi)) =2g$ (we refer to Subsection \ref{invariantsubspace} for a definition). By \cite{AF1} we know that the set of $(h,[\lambda]) \in H(\pi) \times \mathbb{P}_+^{d-1}$ such that the associate flow is  weakly mixing has full measure. Here we obtain
\begin{theoalph}
The set of $(h,[\lambda]) \in H(\pi) \times \mathbb{P}_+^{d-1}$ such that the associate translation flow is not weakly mixing has Hausdorff dimension strictly less than $2g+d-1$.
\end{theoalph} 

\textbf{Acknowledgements: }we would like to thank S\'ebastien Gouëzel for his very careful reading of a first version of this paper and many useful comments. The second author is also grateful to Julie D\'eserti and Harold Rosenberg for their constant support, and thanks \textit{\'Ecole polytechnique} and \textit{Réseau Franco-Brésilien de Mathématiques} for their financial support during several visits at Instituto de Matem\'atica Pura e Aplicada in Rio de Janeiro.

\section{Outline}

The property of weak mixing we are interested in concerns the dynamics of some i.e.t.\ $f$ on the interval $I$, or in other terms, \textit{phase space}. But $f$ is parametrized by $([\lambda],\pi) \in \mathbb{P}_+^{d-1} \times \mathfrak{S}_d^0$ and we will see that it is possible to define a dynamics on the \textit{space of parameters} as well. The so-called \textit{Veech criterion} gives a link between the property of weak mixing for phase space and the dynamics of some cocycle in parameter space.

Indeed, ``bad" parameters, i.e.\ corresponding to i.e.t.'s which are not weakly mixing, can be detected through a cocycle $(T,A)$ derived from the Rauzy \text{cocycle.} For a parameter $[\lambda]$, the weak-stable lamination $W^s([\lambda])$ is defined to be the set of vectors $h$ whose iterates under the cocycle get closer and closer to the lattice $\Z^d$. Veech criterion tells us that the weak-stable lamination $W^s([\lambda])$ associated to some ``bad" parameter $[\lambda]$ contains the element $(t,\dots,t)$ for some $t \in \R \backslash \Z$. Prevalent weak mixing can thus be obtained by ruling out intersections between $\mathrm{Span}(1,\dots,1)\backslash \Z^d$ and $W^s([\lambda])$ for typical $[\lambda]$. In \cite{AF1}, typical meant ``almost everywhere''; following the strategy developed by Delecroix and the first author in \cite{AD}, we show here that this holds actually for every $[\lambda]$ but for a set whose Hausdorff dimension is not maximal, which is stronger.

The study of the weak-stable lamination of $(T,A)$ was done in \cite{AF1}. To achieve this, and given $\delta >0$, $m \in \N$, the authors introduce the set $W_{\delta,m}([\lambda])$ of vectors $h$ with norm less than $\delta$ and such that the iterates $A_k([\lambda]) \cdot h$ remain small for the pseudo-norm $\|\cdot\|_{\R^d/ \Z^d}$ up to time $m$. 

The analysis is based on the following process: given a little segment $J$ near the origin, its image by the cocycle may again contain a point near some element $c \in \Z^d$. When $c=0$, the corresponding line is called a \textit{trivial child} of $J$. Else, we translate the image of $J$ by $-c$ to bring it back to the origin, and we call this new segment a \textit{non-trivial child} of $J$. Note that there may be several non-trivial children. The goal of the study is to show that for most $[\lambda]$, this process has finite life expectancy, that is, the family generated by a line is finite. 

A key ingredient in the ``local" analysis near the origin is the existence of two positive Lyapunov exponents for surfaces of genus at least $2$. For a fixed segment not passing through the origin, the biggest Lyapunov exponent is responsible for the growth of the length of its iterates by the cocycle, while the second biggest generates a drift that tends to kick them further away from the origin. 

The second part of the argument corresponds to a ``global" analysis which handles the fact that points near the origin may become close to another integer element under cocycle iteration. To address this point, Forni and the first author have developed a probabilistic argument. Choose a finite set $S$ of matrices such that for a typical parameter $[\lambda]$, the proportion of integers $k \geq 0$ such that $A(T^k([\lambda])) \in S$ is big. Since the set $S$ is finite, it is possible to ensure that when the matrix we apply belongs to $S$, the only potential child is trivial, and moreover, such an element kicks each line further from the origin by at least a given factor. It follows that for a typical $[\lambda]$ and for every line $J$, the process has finite life expectancy: for every $\delta >0$ and every line $J$, there exists some integer $m\geq 0$ such that $J \cap W_{\delta,m}^{s}([\lambda]) = \emptyset$.

In the present paper, we adapt the estimates obtained in \cite{AF1} to show that for every line $J$, the measure of the set of $[\lambda]$ such that the process survives up to time $m$ goes to zero exponentially fast with respect to $m$: if $\Gamma_{\delta}^m(J)$ denotes the set of $[\lambda]$ such that $J \cap W_{\delta,m}^s ([\lambda]) \neq \emptyset$, we get that $\mu(\Gamma_\delta^m(J)) \leq C e^{-\kappa m} \|J\|^{-\rho}$ for some constants $C, \kappa,\rho > 0$. We give here a proof of this fact using a general large deviations result obtained in \cite{AD}.

But by \cite{AD}, we know that the decay of the volumes $\mu(\Gamma_\delta^m(J))$ can be converted into a bound on the Hausdorff dimension of $\cap_m \Gamma_\delta^m(J)$. Since these sets are intimately related to the weak-stable lamination, we thus get a control of the Hausdorff dimension of ``bad'' parameters $[\lambda]$, which concludes.

For the case of translation flows, we use the parametrization $(h,[\lambda])$ that comes from Veech's zippered rectangle construction. For a fixed parameter $h$ and any integer $m\geq 0$, as for interval exchange transformations, we can estimate the measure of the set of ``bad'' parameters $[\lambda]$ at time $m$; this allows us to construct a cover of this set of parameters. Then, if we consider some $h'$ exponentially close to $h$, then for any ``bad'' parameter $[\lambda']$, we have two cases: either the iterates of the cocycle grow very fast (by some large deviations result, this happens very rarely), or $[\lambda']$ belongs to some piece of the cover constructed for $h$; in other terms, the previous cover can be taken locally constant with respect to $h$. Thus for each integer $m\geq 0$, we obtain a cover of ``bad'' parameters $(h,[\lambda])$ at time $m$, and we can then estimate the Hausdorff dimension of non weakly mixing parameters. 

\section{Background}\label{background}

As we have explained, the property of weak mixing for i.e.t.'s is related to the dynamics of a cocycle $(T,A)$ defined in parameter space. An important fact is that the map $T$ we consider has a property called \textit{bounded distortion}; indeed it is crucial for the probabilistic argument that we outlined, in particular to get some large deviations result. 
In this section, we will see that under some assumptions, it can be checked when the map $T$ is obtained by restriction to a simplex compactly contained in the projective space. Another important point is that the cocycle $(T,A)$ is \textit{fast decaying}; we will recall this notion and see how it can be used to give a bound on the Hausdorff dimension of certain sets. 

\subsection{Strongly expanding maps}\label{backstrong}

Let $(\Delta,\mu)$ be a probability space and
$T\colon\Delta \to \Delta$ a measurable transformation preserving the measure class of $\mu$. We say $T$ is {\it weakly
expanding }if there exists a partition (modulo $0$) $\{\Delta^{(l)}, \,l \in \Z\}$ of $\Delta$
into sets of positive measure, such that for all $l \in \Z$, $T$ maps $\Delta^{(l)}$
onto $\Delta$, $T^{(l)}:=\restriction{T}{\Delta^{(l)}}$ is invertible and the pull-back $(T^{(l)})^* \mu$ is equivalent
to $\restriction{\mu}{\Delta^{(l)}}$.

Let $\Omega$ be the set of all finite words with integer entries. Given $\underline{l}=l_0 \dots l_{n-1}\in \Omega$, we denote by $|\underline{l}|:=n$ its length; we also define
$\Delta^{\underline{l}}:=\bigcap\limits_{k=0}^{n-1} T^{-k} \Delta^{(l_k)}$ and $T^{\underline{l}}:=\restriction{T^n}{\Delta^{\underline{l}}}$.  In particular $\mu(\Delta^{\underline{l}})>0$ by weak expansiveness.

If $\underline{l} \in \Omega$, we denote $\mu^{\underline{l}}:=\frac {1} {\mu(\Delta^{\underline{l}})} T^{\underline{l}}_* (\restriction{\mu}{\Delta^{\underline{l}}})$. We say that $T$ is {\it strongly expanding} if for some $K>0$, 
\begin{equation}\label{2.2}
K^{-1} \leq \frac {d\mu^{\underline{l}}} {d\mu} \leq K,\quad \underline{l} \in \Omega.
\end{equation}
\begin{lem}\label{bdddistt}
Let $T$ be strongly expanding and $Y \subset \Delta$ with
$\mu(Y)>0$; the following bounded distortion property holds:
\begin{equation}\label{2.3}
K^{-2} \mu(Y) \leq \frac {T^{\underline{l}}_* (\restriction{\mu^{\underline{l}'}}{\Delta^{\underline{l}}})(Y)} {\mu(\Delta^{\underline{l}})} \leq K^2 \mu(Y),
\quad \underline{l},\underline{l}' \in \Omega.
\end{equation}
\end{lem}

\subsection{Projective transformations}\label{projjj}

We let $\mathbb{P}^{d-1}_+ \subset \mathbb{P}^{d-1}$ be the projectivization of $\R^d_+$. A {\it projective contraction} is the projectivization of some
matrix $B \in \mathrm{GL}(d,\R)$ with non-negative entries; in particular, the associate transformation takes $\mathbb{P}^{d-1}_+ $ into
itself.
The image of $\mathbb{P}^{d-1}_+$ by a projective contraction is
called a {\it simplex}. 

\begin{lem}[Lemma 2.1, Avila-Forni \cite{AF1}]\label{projectivestronglyexpanding}

Let $\Delta$ be a simplex compactly contained in $\mathbb{P}^{d-1}_+$
and $\{\Delta^{(l)}\}_{l \in \Z}$ a partition of $\Delta$ into sets of positive Lebesgue measure. Let
$T\colon\Delta \to \Delta$ be a measurable transformation  such that, for all $l \in
\Z$, $T$ maps $\Delta^{(l)}$ onto $\Delta$, $T^{(l)}:=\restriction{T}{\Delta^{(l)}}$ is invertible
and its inverse is the restriction of a projective contraction.  Then $T$ preserves
a probability measure $\mu$ which is absolutely continuous with respect to Lebesgue
measure and has a density which is continuous and positive in $\overline \Delta$.
Moreover, $T$ is strongly expanding with respect to $\mu$.

\end{lem}

\subsection{Cocycles} \label {cocycle}

Let $(\Delta,\mu)$ be a probability space. A {\it cocycle} is a pair $(T,A)$, where
$T\colon\Delta \to \Delta$ and $A\colon\Delta \to \mathrm{GL}(d,\R)$ are measurable maps; it can be viewed as a linear skew-product $(x,w) \mapsto (T(x),A(x) \cdot w)$ on
$\Delta \times \R^d$.  If $n\geq 0$ we have $(T,A)^n=(T^n,A_n)$,  where
\begin{equation*}
A_n(x):=A(T^{n-1}(x)) \cdots A(x).
\end{equation*}
We say that $(T,A)$ is \textit{integral} if $A(x) \in \mathrm{GL}(d,\Z)$ for $\mu$-almost every $x \in \Delta$.

Assume that $T\colon\Delta \to \Delta$
is strongly expanding with respect to a partition $\{\Delta^{(l)}\}_{l \in \Z}$ of $\Delta$, and that the $\sigma$-algebra of $\mu$-measurable sets is generated ($\mathrm{mod}\ 0$) by the $\Delta^{\underline{l}}$'s. For $n \geq 0$, we define $\mu_n:=\frac{1}{n} \sum_{k=0}^{n-1} T_*^{k} \mu$ and take $\nu$ a weak-star limit of $(\mu_n)$. Then $\nu$ is an ergodic probability measure which is invariant by $T$. 

Given $B \in \mathrm{GL}(d,\R)$, we define $\|B\|_0:=\max \{\|B\|,\|B^{-1}\|\}$. The cocycle $(T,A)$ is \textit{log-integrable} if 
\begin{equation}\label{uniform}
\int_\Delta \ln \|A(x)\|_0 d\nu(x)<\infty.
\end{equation}

We say that $(T,A)$ is {\it locally constant} if for all $l\in \Z$, $\restriction{A}{\Delta^{(l)}}$ is a constant $A^{(l)}$. In this
case,  for all ${\underline l}\in \Omega$, $\l=l_1\dots l_n$, we set
\begin{equation*}
A^{\underline{l}}:=A^{(l_n)} \cdots A^{(l_1)}.
\end{equation*}

\subsection{Fast decay}\label{fastdecaydefin}

Let $(\Delta,\mu)$ be a probability space. Assume that $T$ is weakly expanding with respect to a partition $\{\Delta^{(l)}\}_{l \in \Z}$ of $\Delta$ and that the cocycle $(T,A)$ is locally constant. As in \cite{AD}, $T$ is \textit{fast decaying} if there exist $C_1 >0$, $\alpha_1 >0$ such that
\begin{equation}\label{fastdecay1}
\sum\limits_{\mu(\Delta^{(l)}) \leq \varepsilon} \mu(\Delta^{(l)}) \leq C_1 \varepsilon^{\alpha_1},\quad 0 < \varepsilon < 1,
\end{equation}
and we say that $A$ is \textit{fast decaying} if there exist $C_2 >0$, $\alpha_2 >0$ such that
\begin{equation}\label{fastdecay2}
\sum\limits_{\|A^{(l)}\|_0 \geq n} \mu(\Delta^{(l)}) \leq C_2 n^{-\alpha_2}.
\end{equation}
In particular, fast decay of $A$ implies that the cocycle $(T,A)$ is log-integrable. If both $T$ and $A$ are fast decaying we say that the \textit{cocycle} $(T,A)$ is fast decaying.

\subsection{Hausdorff dimension}

Let $X$ be a subset of a metric space $M$. Given $d \in \R_+$, its $d-$dimensional Hausdorff measure is defined as follows:
\begin{equation}
\mu_d(X)=\lim\limits_{\varepsilon \to 0}\inf\limits_{\{U_i^\varepsilon\}} \sum\limits_{i} \mathrm{diam}(U_i)^d,
\end{equation}
where the infimum is taken over all countable covers $\{U_i^\varepsilon\}$ of $X$ such that $\mathrm{diam}(U_i^\varepsilon) < \varepsilon$ for all $i$.

\begin{defi}
The \textit{Hausdorff dimension} of $X$ is the unique value $d=:\mathrm{HD}(X) \in \R_+ \cup \{\infty\}$ such that $\mu_{d'}(X)=0$ if $d' > d$ and $\mu_{d'}(X)=\infty$ if $d' < d$.
\end{defi}

\subsection{Fast decay \& Hausdorff dimension}\label{fastdecayhausdorffdim}

Let $\Delta \Subset \mathbb{P}_+^{d-1}$ be a simplex, and assume that $T\colon\Delta \to
\Delta$ satisfies the hypotheses of Lemma \ref{projectivestronglyexpanding}. 

\begin{thm}[Theorem 27, Avila-Delecroix \cite{AD}] \label{thm:HD}
Assume that $T$ is fast decaying and take $\alpha_1 >0$ as in (\ref{fastdecay1}). For $n \geq 1$, let $X_n \subset \Delta$ be
a union of $\Delta^{\underline{l}}$ with $|{\underline{l}}|=n$,
and define $X:=\liminf\limits_{n \to \infty} X_n$.  If
\begin{equation*}
\delta:= \limsup_{n \to \infty} - \frac {1} {n} \ln \mu(X_n) > 0,
\end{equation*}
then
\begin{equation*}
\mathrm{HD}(X) \leq d-1-\min(\delta,\alpha_1) < d-1.
\end{equation*}
\end{thm}

\section{Interval exchange transformations and renormalization algorithms}\label{sectionintervalexchangerenormalization}

Following the notations of \cite{MMY} and \cite{AGY}, we recall some classical notions of the theory of interval exchange transformations (see also \cite{Via}, \cite{V4}). In particular, we give the definition of Rauzy induction and renormalization procedures. The rough idea is the following: given an i.e.t.\ $f$, we look at the first-return map induced by $f$ on some subintervals that are chosen smaller and smaller. This allows us to accelerate the dynamics in phase space in order to capture asymptotic behaviors such as weak mixing. But each return map is itself an i.e.t., and the corresponding changes of parameters define a dynamics in parameter space. We also recall some classical notions on translation surfaces and translation flows, which are a continuous counterpart to i.e.t's. By Veech's ``zippered rectangles" construction, it is possible to suspend any i.e.t.\ to a flow on a translation surface, which is obtained by gluing rectangles on each subinterval and performing certain identifications between them. In the space of ``zippered rectangles'', the extension of Rauzy induction can be seen as a cocycle, called the Rauzy cocycle. Similarly it is possible to define a cocycle over Rauzy renormalization map, and we will see that by considering first-return maps to a simplex compactly contained in $\mathbb{P}_+^{d-1}$, it induces a cocycle $(T,A)$ with better properties: indeed, $T$ has bounded distortion and $(T,A)$ is fast decaying. 

\subsection{Interval exchange transformations}\label{ietts}

Let $\mathcal{A}$ be an alphabet on $d > 1$ letters, and let $I \subset \R$ be an interval having $0$ as left endpoint. We choose a partition $\{I_\alpha\}_{\alpha\in \mathcal{A}}$ of $I$ into subintervals which we assume to be closed on the left and open on the right. In the following, we denote $\R_+^{\mathcal{A}}\sim \R_+^{d}$ and $\mathbb{P}_+^{\mathcal{A}}:=\mathbb{P}(\mathbb{R}_+^{\mathcal{A}})\sim \mathbb{P}_+^{d-1}$. An \textit{interval exchange transformation}, or \textit{i.e.t.}, is a bijection of $I$ defined by two data:
\begin{enumerate}
\item A vector $\lambda=(\lambda_\alpha)_{\alpha\in \mathcal{A}} \in \R_+^{\mathcal{A}}$ whose coordinates correspond to the lengths of the subintervals: for every $\alpha \in \mathcal{A}$, $\lambda_\alpha:=|I_\alpha|$. We also define $|\lambda|:=\sum\limits_{\alpha \in \mathcal{A}} \lambda_\alpha$, so that $I=I^\lambda:=[0,|\lambda|)$.
\item A pair $\pi=\left(
             \begin{array}{c}
               \pi_t \\
               \pi_b \\
             \end{array}
           \right)$ of bijections $\pi_*\colon \mathcal{A} \to \{1,\dots,d\}$, $*=t,b$, prescribing in which way the subintervals $I_\alpha$ are ordered before and after the application of the map. The bijections $\pi_*$ can be viewed as one top and one bottom rows, where the elements of $\mathcal{A}$ are displayed in the order $(\pi_*^{-1}(1),\dots,\pi_*^{-1}(d))$:
    \begin{equation*}
    \pi = \left(
            \begin{array}{cccc}
              \alpha_1^t & \alpha_2^t & \dots & \alpha_d^t \\
              \alpha_1^b & \alpha_2^b & \dots & \alpha_d^b \\
            \end{array}
          \right).
    \end{equation*}
We sometimes identify $\pi$ with its \textit{monodromy invariant} $\tilde{\pi}:=\pi_b \circ \pi_t^{-1}$ and call it a \textit{permutation}. We denote by $\mathfrak{S}(\mathcal{A})$ the set of all such permutations, and by $\mathfrak{S}^0(\mathcal{A})$ the subset of {\it irreducible} ones, that is $\pi \in \mathfrak{S}^0(\mathcal{A})$ if and only if for every $1 \leq k < d$, the set of the first $k$ elements in the top and in the bottom rows do not coincide. 
\end{enumerate}
Given a permutation $\pi \in \mathfrak{S}^0(\mathcal{A})$, and for $*=t,b$, we define linear maps $\Omega_\pi^*\colon \R^{\mathcal{A}} \to \R^{\mathcal{A}}$ by 
\begin{equation}\label{definitionomega}
(\Omega^*_\pi(\lambda))_\alpha:=\sum\limits_{\pi_*(\beta) < \pi_*(\alpha)} \lambda_\beta,\quad \lambda \in \R^{\mathcal{A}},\ \alpha \in {\mathcal A}.
\end{equation}
Set $\Omega_\pi:=\Omega_\pi^b - \Omega_\pi^t$. For any $\lambda \in \R_+^{\mathcal{A}}$, the interval exchange transformation $f=f(\lambda,\pi)$ is the map associated with the translation vector
$w:=\Omega_\pi(\lambda)$; in other terms,
\begin{equation*}
f(x):=x+w_\alpha,\quad x \in I_\alpha.
\end{equation*}
Two i.e.t.'s obtained one from another by a dilation on the length parameter $\lambda$ have the same dynamical behavior; therefore, one can projectivize $\lambda$ to $[\lambda]\in \mathbb{P}_+^{\mathcal{A}}$ and consider $f([\lambda],\pi)$. 

\subsection{Translation surfaces}\label{translsur}

A \textit{translation surface} is a compact Riemann surface $S$ endowed with some nonzero Abelian differential $\omega$. Let $\Sigma \subset S$ be the set of zeros, or \textit{singularities} of $\omega$. For each $s \in \Sigma$, denote by $\kappa_s$ the order of $s$ as a zero. 
For any $p \in S \backslash \Sigma$, there exists a chart defined in the neighborhood of $p$ such that in these coordinates, $\omega$ simply writes down as $dz$. The family of such charts on $S \backslash \Sigma$ forms an atlas for which transition maps correspond to translations in $\R^2$. Moreover, every singularity $s$ has a punctured neighborhood isomorphic via a holomorphic map to a finite cover of a punctured disk in $\R^2$, and such that in this chart $\omega$ becomes $z^{\kappa_s} dz$.

The form $|\omega|$ defines a flat metric on $S$ with conical singularities at $\Sigma$. The total angle around a singularity $s$ is $2 \pi (\kappa_s+1)$. The total area of the surface is given by $\int |\omega|^2 < \infty$. \textit{Normalized} translation surfaces are those for which $\int |\omega|^2 = 1$.

For each
$\theta \in \R / 2\pi\Z$, the \textit{directional flow} in the
direction $\theta$ is the flow $\phi^{S,\theta}_{t}\colon S \rightarrow S$
obtained by integration of the unique vector field $X_\theta$ such
that $\omega(X_\theta) = e^{\mathrm{i} \theta}$.  In local charts, $\omega = dz$ and we have $\phi^{S,\theta}_t(z) = z + t e^{\mathrm{i}\theta}$ for small
$t$, so directional flows are also called {\it translation flows}. The \textit{(vertical) flow} of $(S,\omega)$ is the flow
$\phi^{S,\pi/2}$. Translation flows are not defined at the zeros of
$\omega$ and hence not defined for all positive
times on backward
orbits of the singularities. The flows $\phi^{S,\theta}_t$
preserve the volume
form $\frac{\mathrm{i}}{2} \omega \wedge \overline{\omega}$ and the ergodic
properties of translation
flows we will discuss are with respect to this measure. 

Let us recall some results which hold for an arbitrary translation surface:
the directional flow is minimal except for a countable set of directions
\cite{Ke}, the translation flow is uniquely ergodic except for a
set of directions of Hausdorff dimension at most $1/2$
\cite{KMS}, \cite{Ma2}, and the translation flow
is not mixing in any direction \cite {Ka}.

It is known that for a genus one translation surface, translation flows are never weakly mixing. The same property holds for the branched coverings of genus one translation surfaces, which
form a \textit{dense} subset of translation surfaces.  However, Forni and the first author \cite{AF1} have proved that for
\textit{almost every}
translation surface of genus at least two, the translation flow
is weakly mixing in almost every direction.

Considering translation surfaces of genus $g$ modulo isomorphism, one gets the \textit{moduli space} of Abelian differentials, denoted by $\mathcal{M}_g$. It is possible to define a flow $(g_t)_{t \in \R}$ on this space, called the \textit{Teichmüller flow}: its action on an Abelian differential $\omega=\Re(\omega) +\mathrm{i} \Im(\omega)$ is given by $g_t \cdot \omega:=e^t \Re(\omega) +\mathrm{i} e^{-t}\Im(\omega)$. By fixing the order of zeros as an unordered list $\kappa$ of positive integers, one defines \textit{strata} $\mathcal{M}_{g,\kappa} \subset \mathcal{M}_g$. We also denote by $\mathcal{M}_{g,\kappa}^1 \subset \mathcal{M}_{g,\kappa}$ the hypersurface corresponding to normalized surfaces. Each stratum is an orbifold of finite dimension. For each connected component $\mathcal{C}$ of some $\mathcal{M}_{g,\kappa}^1$, there is a well-defined probability measure $\mu_{\mathcal{C}}$ in the Lebesgue measure class which is invariant by the Teichmüller flow; it is called the \textit{Masur-Veech measure}. The implicit measure-theoretical notions above refer to this measure.

Given a translation surface $(S,\omega)$, a \textit{separatrix} is a geodesic line for the metric $|\omega|$ starting from a singularity in $\Sigma$; it is called a \textit{saddle connection} when the separatrix connects two singularities and has its interior disjoint from $\Sigma$. The first-return map to some separatrix for the vertical flow is an interval exchange transformation. Conversely, it is possible to suspend any interval exchange transformation to a translation flow by a construction we now briefly recall.

\subsection{Veech's ``zippered rectangles'' construction}\label{invariantsubspace}

Let $\mathcal{A}$ be an alphabet on $d>1$ letters. Given $(\lambda,\pi) \in \R_+^{\mathcal{A}}\times \mathfrak{S}^0(\mathcal{A})$, Veech's construction allows to suspend the i.e.t.\ $f(\lambda,\pi)$ to a suspension flow on a translation surface $S$. We consider the convex cone 
$T^+(\pi):=\left\{\tau\in \R^\mathcal{A}\ \left|\ \sum\limits_{\pi_t(\beta)\leq k} \tau_\beta>0\ \text{and}\ \sum\limits_{\pi_b(\beta)\leq k} \tau_\beta<0,\ 1 \leq k \leq d-1\right\} \right.$.
Let $\Omega_\pi$ be the map defined in Subsection \ref{ietts}, and set $H^+(\pi):=-\Omega_\pi(T^+(\pi))\subset \R_+^{\mathcal{A}}$. 
For any $h \in H^+(\pi)$ and $\alpha\in \mathcal{A}$, we define rectangles above (resp. below) top (resp. bottom) subintervals by:
$$
R_\alpha^t:=(w_\alpha^t,w_\alpha^t+\lambda_\alpha) \times [0,h_\alpha]\quad \text{and}\quad R_\alpha^b:=(w_\alpha^b,w_\alpha^b+\lambda_\alpha) \times [-h_\alpha,0],
$$ where $w^*:=\Omega_\pi^*(\lambda)$, $*=t,b$. 
The surface $S$ is obtained by performing appropriate gluing operations on the union of those rectangles. The initial i.e.t.\ $f$ corresponds to the first-return map to the transversal $I$ of the vertical flow in $S$. 

Denote by $\tilde \pi$ the monodromy invariant, and let $\sigma_\pi$ be the permutation on $\{0,\dots,d\}$
defined by
\begin{equation*}
\sigma_\pi(i):=\left \{ \begin{array}{ll}
\tilde \pi^{-1}(1)-1, & i=0,\\[5pt]
d, & i=\tilde \pi^{-1}(d),\\[5pt]
\tilde \pi^{-1}(\tilde \pi(i)+1)-1, & i \neq 0, \tilde \pi^{-1}(d).
\end{array}
\right.
\end{equation*}
For $i,j\in \{0,\dots,d\}$, we denote $i\sim j$ if $i$ and $j$ belong to the same orbit under $\sigma_\pi$. Then the quotient space $\Sigma(\pi):=\{0,\dots,d\}/\sim$ is in one-to-one correspondence with the set of singularities of $S$. 
Moreover, for every $s\in \Sigma(\pi)$,  let $b^s \in \R^d$ be the vector defined by
\begin{equation*}
b^s_i:=\chi_s(i-1)-\chi_s(i), \quad 1 \leq i \leq d,
\end{equation*}
where $\chi_s$ denotes the characteristic function of $s$. We define $\Upsilon(\pi):=\{b^s,\ s \in \Sigma(\pi)\}$. 
 Let us recall the following result.

\begin{lem}[Veech, \cite{V4}, \textsection 5]\label{easylemma}
Let $\pi \in \mathfrak{S}^0(\mathcal{A})$. For each $s \in \Sigma(\pi)$, one has
\begin{equation*}
(1,\dots,1) \cdot b^s=\left \{ \begin{array}{ll}
1, & 0 \in s,\ d \not\in s,\\[5pt]
-1, & 0 \not\in s,\ d \in s,\\[5pt]
0, & \mathrm{otherwise}.
\end{array}
\right.
\end{equation*}
\end{lem}

We define $H(\pi):=\Omega_\pi(\R^{\mathcal{A}})=\ker(\Omega_\pi)^{\perp}$.

\begin{prop}
$H(\pi)$ coincides with the annulator of the subspace of $\R^{\mathcal{A}}$ spanned by $\Upsilon(\pi)$:
\begin{equation*}
h \in H(\pi) \Longleftrightarrow h \cdot b^s=0,\ s \in \Sigma(\pi).
\end{equation*}
Moreover, $\dim(H(\pi))= d+1-\#\Sigma(\pi)=2 g(\pi)$, where $g(\pi)$ is the genus of the suspension surface $S$, and $H(\pi)$ can be identified with the absolute homology $H_1(S,\mathbb{R})$ of $S$.
\end{prop}


\subsection{Rauzy classes} \label {rauzy c}

Let $\pi \in \mathfrak{S}^0(\mathcal{A})$. We denote by $\alpha(t)$ (resp.
$\alpha(b)$) the last element of the top (resp. bottom) row.
We define two bijections of $\mathfrak{S}^0(\mathcal{A})$ as follows; the reason why we consider these transformations will become clear in the next subsection. The {\it top} operation $\underline{t}$ maps $\pi$ to the permutation
\begin{equation*}
\underline{t}(\pi)=\left(
                \begin{array}{cccccccc}
                    \alpha_1^t & \dots & \alpha_{k-1}^t & \alpha_k^t & \alpha_{k+1}^t & \dots & \dots & \alpha(t)\\
                    \alpha_1^b & \dots & \alpha_{k-1}^b & \alpha(t) & \alpha(b) & \alpha_{k+1}^b & \dots & \alpha_{d-1}^{b}\\
                \end{array}
             \right),
\end{equation*}
while the {\it bottom} operation $\underline{b}$ maps $\pi$ to
\begin{equation*}
\underline{b}(\pi)=\left(
                \begin{array}{cccccccc}
                    \alpha_1^t & \dots & \alpha_{k-1}^t & \alpha(b) & \alpha(t) & \alpha_{k+1}^t & \dots & \alpha_{d-1}^t\\
                    \alpha_1^b & \dots & \alpha_{k-1}^b & \alpha_{k}^b & \alpha_{k+1}^b & \dots & \dots & \alpha(b)\\
                \end{array}
             \right).
\end{equation*}
We define a \textit{Rauzy class} to be a minimal non-empty subset of $\mathfrak{S}^0(\mathcal{A})$ which is
invariant under $\underline{t}$ and $\underline{b}$. Given a Rauzy class  $\mathfrak{R}$, the corresponding {\it Rauzy diagram} has its vertices in $\mathfrak{R}$ and is formed by arrows mapping $\pi \in \mathfrak{R}$ to $\underline{t}(\pi)$ or $\underline{b}(\pi)$. The set of all paths in this diagram is denoted by $\Pi(\mathfrak{R})$.

\subsection{Rauzy induction and renormalization procedures}\label{rauzyinductionsec}\label{sectionrauzyrenor}

We now recall the definition of two procedures first introduced by Rauzy in \cite{R} (see also
Veech \cite{V1}). Let $\mathfrak{R} \subset \mathfrak{S}^0(\mathcal{A})$ be some Rauzy class, and let $(\lambda,\pi)\in\R^\mathcal{A}_+ \times \mathfrak{R}$. We denote by $\alpha(t)$ (resp.
$\alpha(b)$) the last element of the top (resp. bottom) row of $\pi$. Assume that $\lambda_{\alpha(t)} \neq \lambda_{\alpha(b)}$ and set $\ell:=|\lambda| - \min(|\lambda_{\alpha(t)}|,|\lambda_{\alpha(b)}|)$. The first-return map of $f(\lambda,\pi)$ to the subinterval $[0,\ell) \subset I^{\lambda}$ is again an i.e.t.; it is the map $f(\lambda^{(1)},\pi^{(1)})$, where the parameters $(\lambda^{(1)},\pi^{(1)})\in \mathbb{R}_+^{\mathcal{A}}\times \mathfrak{R}$ are defined as follows:
\begin{enumerate}
\item If $\lambda_{\alpha(t)}>\lambda_{\alpha(b)}$ (resp. $\lambda_{\alpha(b)}>\lambda_{\alpha(t)}$), we let $\gamma(\lambda,\pi)$ be the top (resp. bottom) arrow starting at $\pi$, and we set $\alpha:=\alpha(t)$ (resp. $\alpha:=\alpha(b)$). 
\item Let $\lambda^{(1)}_\xi:=\lambda_\xi$ if $\xi\neq \alpha$; else, let $\lambda^{(1)}_\xi:=|\lambda_{\alpha(t)}-\lambda_{\alpha(b)}|$.
\item $\pi^{(1)}$ is the end of the arrow $\gamma(\lambda,\pi)$.
\end{enumerate}
This motivates \textit{a posteriori} the introduction of the top and bottom operations.
The map $\mathcal{Q}_R\colon(\lambda,\pi) \mapsto
(\lambda^{(1)},\pi^{(1)})$ defines a dynamical system in parameter space, and is called the {\it Rauzy induction map}.\\

Since $\mathcal{Q}_R$ commutes with dilations on $\lambda$,
it projectivizes to a map $\mathcal{R}_R\colon \mathbb{P}^{\mathcal{A}}_+ \times \mathfrak{R} \to \mathbb{P}^{\mathcal{A}}_+ \times \mathfrak{R}$, called the \textit{Rauzy renormalization map}.
For $n \geq 1$, the connected components of the domain of definition of $\mathcal{R}_R^n$ are naturally labeled by paths in $\Pi(\mathfrak{R})$ of length $n$. Moreover, if $\gamma$ is a path of length $n$ which ends at $\pi^{(n)}\in \mathfrak{R}$, and $D_\gamma$ denotes the associate connected component, then $\restriction{\mathcal{R}_R^n}{D_\gamma}$ follows the path $\gamma$ in the Rauzy diagram and maps $D_\gamma$ homeomorphically to $\mathbb{P}_+^{\mathcal{A}} \times \{\pi^{(n)}\}$. A sufficient condition for $([\lambda],\pi)$ to belong to the domain of $\mathcal{R}_R^n$ for every $n \geq 1$ is for the coordinates of $[\lambda]$ to be independent over $\Q$. We will always assume it implicitely in the following. Since the set of rationally dependent $[\lambda] \subset \mathbb{P}_+^{\mathcal{A}}$ has dimension $d-2$, it is not a limitation for our estimates.

\begin{thm}[Masur \cite {Ma1}, Veech \cite{V2}]\label{invmeasureRauzy}

Let $\mathfrak{R}\subset \mathfrak{S}^0(\mathcal{A})$ be a Rauzy class. Then $\restriction{\mathcal{R}_R}{\mathbb{P}^{\mathcal{A}}_+
\times \mathfrak{R}}$ admits an ergodic conservative infinite absolutely continuous invariant
measure $\mu$, unique in its measure class up to a scalar multiple. Its density is
a positive rational function.

\end{thm}

\subsection{The Rauzy cocycle}

We denote by $E_{\alpha \beta}$ the elementary matrix $(\delta_{i \alpha}\delta_{j \beta})_{1\leq i,j \leq d}$. Let $\mathfrak{R} \subset \mathfrak{S}^0(\mathcal{A})$ be a Rauzy class. We associate with any path $\gamma \in \Pi(\mathfrak{R})$ a matrix $B_\gamma \in \mathrm{SL}(\mathcal{A},\Z)$. The definition is by induction:
\begin{itemize}
\item If $\gamma$ is a vertex, we set $B_\gamma:=\mathbf{1}_d$.
\item If $\gamma$ is an arrow labeled by $\underline{t}$, we define $B_\gamma:=\mathbf{1}_d+E_{\alpha(b) \alpha(t)}$.
\item If $\gamma$ is an arrow labeled by $\underline{b}$, we define $B_\gamma:=\mathbf{1}_d+E_{\alpha(t) \alpha(b)}$.
\item If $\gamma$ is obtained by concatenation of the arrows $\gamma_1,\dots,\gamma_m$, then set $B_\gamma:=B_{\gamma_m}\dots B_{\gamma_1}$.
\end{itemize}

Let us stress the following useful fact. Assume that $(\lambda,\pi)\in \R_+^{\mathcal{A}} \times \mathfrak{R}$ belongs to the domain of $\mathcal{Q}_R^n$, $n \geq 1$, and that the application of $\mathcal{Q}_R^n$ follows the path $\gamma$. Set $\mathcal{Q}_R^n(\lambda,\pi):=(\lambda^{(n)},\pi^{(n)})$; then
\begin{equation}\label{eqqqr}
\lambda^{(n)}=(B_\gamma^*)^{-1} \cdot \lambda.
\end{equation}
In particular, if $\gamma$ starts at $\pi$, then we have $D_\gamma=(B_\gamma^* \cdot \mathbb{P}_+^{\mathcal{A}}) \times \{\pi\}$.

With the notations of Subsection \ref{sectionrauzyrenor}, we define $B^R(\lambda,\pi):=B_{\gamma(\lambda,\pi)}$. Recall that $H^+(\pi):=-\Omega_\pi(T_{\pi}^+)$. Given $(\lambda,\pi)\in \R_+^{\mathcal{A}} \times \mathfrak{R}$ and $h \in H^+(\pi)$, the map $\mathcal{Q}_R$ extends in the following way:
\begin{equation}\label{defrauzz}
\hat{\mathcal{Q}}_R(\lambda,\pi,h):=(\mathcal{Q}_R(\lambda,\pi),B^R(\lambda,\pi) \cdot h).
\end{equation}
It describes the way ``zippered rectangles'' are reordered after application of Rauzy induction; in particular, it leaves the translation structure unchanged. 
Recall that $H(\pi):=\Omega_\pi(\mathbb{R}^{\mathcal{A}})$. It is possible to show that if $\mathcal{Q}_R(\lambda,\pi)=(\lambda^{(1)},\pi^{(1)})$, then
$$
B^R(\lambda,\pi) \cdot H(\pi) = H(\pi^{(1)}).
$$
We thus obtain an integral cocycle $\restriction{B^R(\lambda,\pi)}{H(\pi)}$ over Rauzy induction, called the \textit{Rauzy cocycle}. 

It is easy to see that if $[\lambda']=[\lambda]$, then $\gamma(\lambda',\pi)=\gamma(\lambda,\pi)$, hence the application $([\lambda],\pi)\mapsto B^R([\lambda],\pi)$ is well defined. If $([\lambda],\pi) \in  \mathbb{P}^{\mathcal{A}}_+ \times \mathfrak{R}$, then analogously, the restriction $\restriction{B^R([\lambda],\pi)}{H(\pi)}$ defines a cocycle over $\restriction{\mathcal{R}_R}{\mathbb{P}^{\mathcal{A}}_+ \times \mathfrak{R}}$, which we will also call the Rauzy cocycle. 

\subsection{Recurrence for $\mathcal{R}_R$, bounded distortion and fast decay}\label{rauzyfast}

We choose here $\mathcal{A}=\{1,\dots,d\}$ for some integer $d > 1$, and denote $\mathfrak{S}_d^0:=\mathfrak{S}^0(\{1,\dots,d\})$. Let $\mathfrak{R} \subset \mathfrak{S}_d^0$ be a Rauzy class and choose $\pi \in \mathfrak{R}$. Assume that for some path $\gamma_0\in \Pi(\mathfrak{R})$ which starts and ends at $\pi$, the coefficients of the matrix $B_{\gamma_0}$ are all positive. We let $\Delta:=B_{\gamma_0}^* \cdot \mathbb{P}_+^{d-1}$; it is a simplex compactly contained in $\mathbb{P}_+^{d-1}$, that we identify here with $\{\lambda \in \mathbb{R}_{+}^d,\ | \lambda |:=\sum_i \lambda_i=1\}$.
By projection on the first coordinate, the first-return map of $\mathcal{R}_R$ to $\Delta \times \{\pi\}$ induces a map $T\colon \Delta \to \Delta$, whose domain of definition we denote by $\Delta^1$. By Poincaré recurrence theorem, we know that $\Delta^{1}$ has full measure inside $\Delta$. 

Let us denote by $\Pi(\mathfrak{R})_\pi$ the subset of paths in $\Pi(\mathfrak{R})$ that start and end at $\pi$, and which are \textit{primitive} in the sense that all intermediate vertices differ from $\pi$. Then $\Delta^1$ admits a countable partition $(\Delta_\gamma)_{\gamma \in \Pi(\mathfrak{R})_\pi}$, where the subset $\Delta_\gamma:=B_\gamma^* \cdot \Delta$ corresponds to those $\lambda \in \Delta^1$ for which $(\lambda,\pi)\in D_\gamma$ first returns to $\Delta \times \{\pi\}$ under $\mathcal{R}_R$ after having followed the path $\gamma$ in the Rauzy diagram. In particular $\restriction{T}{\Delta_\gamma}\colon \lambda \mapsto \frac{(B_\gamma^*)^{-1} \cdot \lambda}{|(B_\gamma^*)^{-1} \cdot \lambda|}$. In the following we will denote $h_\gamma:=(\restriction{T}{\Delta_\gamma})^{-1}$.

\begin{lem}\label{boundddd}
The map $T$ preserves a probability measure $\mu$ with respect to which it is strongly expanding, hence has bounded distortion.
\end{lem}

\begin{proof}
We have seen that $\Delta \Subset \mathbb{P}_+^{d-1}$ admits a countable partition (modulo $0$) $\{\Delta_\gamma\}_{\gamma \in \Pi(\mathfrak{R})_\pi}$ into subsets $\Delta_\gamma=B_\gamma^* \cdot \Delta$ of positive Lebesgue measure. For each such $\gamma$, $\restriction{T}{\Delta_\gamma}$ maps $\Delta_\gamma$ bijectively onto $\Delta$, and its inverse $h_\gamma$ is the restriction of the projective contraction $B_{\gamma}^*$.
Since $\Delta\Subset \mathbb{P}_+^{d-1}$, Lemma \ref{projectivestronglyexpanding} tells us that the map $T$ preserves a probability measure $\mu$ which is absolutely continuous with respect to Lebesgue measure. Moreover, $T$ is strongly expanding with respect to $\mu$, and it also has bounded distortion by Lemma \ref{bdddistt}.
\end{proof}

\noindent In the following we consider the probability measure $\mu$ given by Lemma \ref{boundddd}

\begin{lem}\label{fastdecaylemma}
The map $T$ is fast decaying.
\end{lem}

Before giving the proof, we need to recall a fact from \cite{AGY}. In this paper, Gouëzel, Yoccoz and the first author investigate the properties of a flow defined as a suspension over $\mathcal{R}_R$. 
The set $\Delta\times \{\pi\}$ can be seen as a transverse section for the flow; moreover, the first-return time function $r_\Delta\colon \Delta \to \mathbb{R}_+\cup\{\infty\}$ of this flow to $\Delta\times \{\pi\}$ satisfies: for any $\gamma \in \Pi(\mathfrak{R})_\pi$,  
\begin{equation}\label{returntm}
r_\Delta \circ h_\gamma\colon \lambda \mapsto \log |B_\gamma^*\cdot \lambda|.
\end{equation}
Their analysis implies the following result.
\begin{theo}[Theorem 4.7, Avila-Gouëzel-Yoccoz \cite{AGY}]
The map $r_\Delta$ has exponential tails, i.e., there exists $\sigma_0 > 0$ such that
\begin{equation*}
A:=\int_\Delta e^{\sigma_0 r_\Delta(\lambda)} d\mu(\lambda) < \infty.
\end{equation*}
\end{theo}
\noindent By Markov inequality, it follows that for any $r \geq 0$,
\begin{equation}\label{markkk}
\mu(\lambda \in \Delta,\ r_\Delta(\lambda) \geq r)\leq A e^{-\sigma_0 r}.
\end{equation}
They also obtain the following distortion estimate.
\begin{lem}[Lemma 4.6, Avila-Gouëzel-Yoccoz \cite{AGY}]
There exists a constant $C >0$ such that for any path $\gamma \in \Pi(\mathfrak{R})_\pi$,
\begin{equation}\label{distrdelta}
\|D(r_\Delta \circ h_\gamma)\|_{C^0} \leq C.
\end{equation}
\end{lem}

\noindent \textit{Proof of Lemma \ref{fastdecaylemma}.}
Let $\gamma \in \Pi(\mathfrak{R})_\pi$. The Jacobian of $h_\gamma$ at the point $\lambda \in \Delta$ is given by $|\det D h_\gamma(\lambda)|=\frac{1}{|B_\gamma^* \cdot \lambda|^d}$. 
We deduce
\begin{equation}\label{eqmessimplexe}
\mu(\Delta_\gamma)=\mu(B_\gamma^* \cdot \Delta)=\int_\Delta |B_\gamma^* \cdot \lambda|^{-d} d \mu (\lambda)=\int_\Delta e^{-d r_\Delta \circ h_\gamma(\lambda)} d\mu(\lambda).
\end{equation}
Let $\varepsilon >0$ and assume that $\mu(\Delta_\gamma) \leq \varepsilon$. If $\lambda_0 \in \Delta$ is chosen such that $r_\Delta \circ h_\gamma$ is maximal in $\lambda_0$, then from (\ref{eqmessimplexe}), we see that 
$$
r_\Delta \circ h_\gamma (\lambda_0) \geq -\frac{\ln(\varepsilon)}{d} + C(\Delta),
$$
where we have set $C(\Delta) := \frac{\ln(\mu(\Delta))}{d}$. 
By (\ref{distrdelta}), we obtain a constant $C'(\Delta)$ such that for any $\lambda \in \Delta$,
$$
r_\Delta \circ h_\gamma (\lambda) \geq -\frac{\ln(\varepsilon)}{d} + C'(\Delta).
$$
If we choose $r:=-\frac{\ln(\varepsilon)}{d} + C'(\Delta)$ in (\ref{markkk}), we thus get
\begin{align*}
\sum\limits_{\mu(\Delta_\gamma) \leq \varepsilon} \mu(\Delta_\gamma)&\leq \mu(\lambda \in \Delta,\ r_\Delta(\lambda) \geq r)\\
&\leq C_1 \varepsilon^{\alpha_1}
\end{align*}
with $C_1:=A e^{-\sigma_0 C'(\Delta)}>0$ and $\alpha_1:=\sigma_0 /d>0$. 
\qed

\noindent Furthermore, the Rauzy cocycle induces the locally constant cocycle $(T,A)$ where $\restriction{A}{\Delta_\gamma}:=B_\gamma$. 
\begin{lem}\label{fastdddd}
The cocycle $(T,A)$ is fast decaying.
\end{lem}

\begin{proof}
It remains to show that $A$ is fast decaying. For any $\gamma \in \Pi(\mathfrak{R})_\pi$, $B_\gamma \in \mathrm{SL}(d,\Z)$; it follows that if $\|B_\gamma\|_\infty\leq M$, then we also have $\|B_\gamma^{-1}\|_\infty\leq (d-1)!M^{d-1}$. In particular, for any $n \geq0$, $\max(\|B_\gamma\|_\infty,\|B_\gamma^{-1}\|_\infty) \geq n$ implies that $\min(\|B_\gamma^{*}\|_{\infty},\|(B_\gamma^{*})^{-1}\|_{\infty}) \geq \left(\frac{n}{(d-1)!}\right)^{\frac{1}{d-1}}$. Therefore, by \eqref{eqmessimplexe}, and by equivalence of the norms, there exists a constant $\widetilde C_0>0$ depending only on $d$, $\mu$ and $\Delta$ such that if $\|B_\gamma\|_0 \geq n$ for some $n \geq 0$, then $\mu(\Delta_\gamma)=\int_\Delta |B_\gamma^* \cdot \lambda|^{-d} d \mu (\lambda) \leq \widetilde C_0 n^{-\frac{d}{d-1}}\leq \widetilde C_0 n^{-1}$. We deduce from what precedes that
\begin{align*}
\sum\limits_{\|B_\gamma\|_0 \geq n} \mu(\Delta_\gamma)&\leq \sum\limits_{\mu(\Delta_\gamma) \leq \widetilde C_0 n^{-1}} \mu(\Delta_\gamma)\\
&\leq C_2 n^{-\alpha_2},
\end{align*}
where we have set $C_2:=C_1 \widetilde{C}_0^{\alpha_1}>0$ and $\alpha_2:=\alpha_1>0$,
which concludes.
\end{proof}

We have seen that for the transformation $T$ to be well defined, we have to restrict ourselves to the (full-measure) subset $\Delta^1\subset \Delta$. Similarly, for every $n \geq 1$, let $\Delta^n$ be the domain of $T^n$ and denote by $\Delta^\infty:=\bigcap\limits_{n \in \N} \Delta^n$ the subset of points in $\Delta$ which come back to $\Delta$ infinitely many times. To conclude this part, we will explain why it is not a limitation to consider only points in $\Delta^{\infty}$, since it will not affect the result on Hausdorff dimension we aim to show. It follows in fact from a result of Delecroix and the first author that we now recall, and which tells us that the set of escaping points has small Hausdorff dimension.

\begin{prop}[Theorem 29, Avila-Delecroix \cite{AD}]\label{escapingpoints} Let $\Delta$ be a simplex in $\mathbb{P}_+^{d-1}$ admitting a partition $\{\Delta^{(l)}\}_{l \in \Z}$, and let $T\colon \Delta \to \Delta$ be a map with bounded distortion and such that for every $l \in \Z$, $\restriction{T}{\Delta^{(l)}}$ is a projective transformation. Assume that $T$ is fast decaying with fast decay constant $\alpha_1 >0$. Then $$\mathrm{HD}(\Delta\backslash\Delta^{\infty}) \leq d-1 - \frac{\alpha_1}{1+\alpha_1} < d-1.$$
\end{prop}

\section{General weak mixing for interval exchange transformations \& translation flows} \label{mainresultsection}

\subsection{Weak mixing for interval exchange transformations}

Let $d > 1$ be an integer and recall that $\mathfrak{S}_d^0:=\mathfrak{S}^0(\{1,\dots,d\})$. Our main result is the following.

\begin{theo}\label{theoprincipal}
Let $\pi \in \mathfrak{S}_d^0$ be an irreducible permutation which is not a rotation; then the set of $[\lambda] \in \mathbb{P}^{d-1}$ such that $f([\lambda],\pi)$ is not weakly mixing has Hausdorff dimension strictly less than $d-1$.
\end{theo}

As recalled in the introduction, weak mixing for the interval exchange transformation $f$ is equivalent to the non-existence of non-constant measurable solutions $\phi\colon I \to \C$ to the following \text{equation:}
\begin{equation}\label{equcontspec}
\phi(f(x))=e^{2 \pi \mathrm{i} t} \phi(x),\quad x \in I,
\end{equation}
for any $t \in \R$. This is equivalent to the following two conditions:
\begin{itemize}
\item $f$ is ergodic;
\item for any $t \in \R \backslash \Z$, there is no nonzero measurable solution $\phi\colon I \to \C$ to the equation
\begin{equation}\label {phit} 
\phi(f(x))=e^{2 \pi \mathrm{i} t} \phi(x),\quad x \in I.
\end{equation}
\end{itemize}

\noindent Let $\pi$ satisfying the assumptions of Theorem \ref{theoprincipal}. The proof follows three steps:
\begin{enumerate}
\item Using fast decay and results of \cite{Ma1} and \cite{V2}, we show that the set of $[\lambda]$ such that $f([\lambda],\pi)$ is not ergodic does not have full Hausdorff dimension.
\item We then adapt a theorem of Veech \cite{V4} to deal with the case where $(1,\dots,1) \notin H(\pi)$ (the genus-one case is a particular instance of it).
\item The case where $g > 1$ and $(1,\dots,1) \in H(\pi)$ is addressed by adapting the probabilistic argument contained in \cite{AF1} to get some estimate on the Hausdorff dimension of ``bad" parameters.
\end{enumerate}

\subsubsection{Unique ergodicity of interval exchange transformations}\label{subsectionergodicity}

Let $\mathfrak{R} \subset \mathfrak{S}_d^0$ be a Rauzy class, let $\pi \in \mathfrak{R}$. Given $\lambda \in \R_+^{d}$ with rationally independent coordinates, we consider  $f=f(\lambda,\pi)\colon I \to I$. Let us denote by $\mathcal{M}_f:=\{\mu,\ f_* \mu = \mu\}$ the set of invariant measures; recall that the map $\mu \mapsto (\mu(I_i))_{i=1,\dots,d}$ is a linear isomorphism between $\mathcal{M}_f$ and the set $$
\bigcap_{n \geq 1} B_n^R(\lambda,\pi)^* \cdot \R_+^d.
$$
From this fact, Veech (see Proposition 2.13, \cite{V1}) deduces that $f([\lambda],\pi)$ is uniquely ergodic as soon as there exists $n \geq 1$ such that all the coefficients of $B_n^R([\lambda],\pi)$ are positive, and for infinitely many $k\geq 0$, $B_n^R(\mathcal{R}_R^{k}([\lambda],\pi))=B_n^R([\lambda],\pi)$. 

The first point is handled in \cite{MMY}, \textsection 1.2.3-1.2.4. In this work the authors show that for an i.e.t.\  which satisfies \textit{Keane's condition}, there exists an integer $n \geq 1$ such that all the coefficients of $B_n^R([\lambda],\pi)$ are positive. We won't detail the definition of Keane's condition; just recall that it is implied by our assumption on the coordinates of $[\lambda]$ to be rationally independent. Moreover, a result due to Keane states that any i.e.t.\ satisfying this condition is topologically minimal, that is, the orbit of any point is dense. 

Take $n \geq 1$ such that all the coefficients of $B_n^R([\lambda],\pi)$ are positive, and let $\gamma_0 \in \Pi(\mathfrak{R})$, $|\gamma_0|=n$, be the corresponding path in the Rauzy diagram. We also assume that $\gamma_0$ ends at $\pi$. Let $\Delta=\Delta_{\gamma_0}:=B_{\gamma_0}^* \cdot \mathbb{P}_{+}^{d-1}$ be the associate simplex; in particular, it is compactly contained in $\mathbb{P}_{+}^{d-1}$ since all the coefficients of $B_{\gamma_0}$ are positive. By expansiveness of Rauzy renormalization (indeed $\mathcal{R}_R^n(\Delta \times \{\pi\})=\mathbb{P}_+^{d-1} \times \{\pi\}$), every point in $\mathbb{P}_+^{d-1} \times \{\pi\}$ is in the forward orbit of a point in $\Delta \times \{\pi\}$, so it is enough to estimate the dimension of parameters which belong to $\Delta \times \{\pi\}$ and that do not come back infinitely many times to it. As in Subsection \ref{rauzyfast}, we consider the map $T\colon \Delta \to \Delta$ induced by $\mathcal{R}_R$ on $\Delta$; we have seen that it has bounded distortion and that it is fast decaying for some constant $\alpha_1 >0$ (see \eqref{fastdecay1} for the definition). From Proposition \ref{escapingpoints}, we know that the set $\Delta \backslash \Delta^\infty$ of $[\lambda] \in \Delta$ whose iterates do not come back to $\Delta$ infinitely many times has Hausdorff dimension less than $d-1-\frac{\alpha_1}{1+\alpha_1}$. Therefore, we deduce from the above criterion that the set of $[\lambda] \in \Delta$ such that the i.e.t.\ $f([\lambda],\pi)$ is not uniquely ergodic has Hausdorff dimension less than $d-1-\frac{\alpha_1}{1+\alpha_1}$. By the previous remark on expansiveness, we have thus obtained: 
\begin{thm}\label{theouniq}
Let $d>1$ and let $\pi \in \mathfrak{S}_d^0$; the set of $[\lambda] \in \mathbb{P}^{d-1}_+$ such that $f([\lambda],\pi)$ is not uniquely ergodic has Hausdorff dimension strictly less than $d-1$.
\end{thm}

\subsubsection{Veech criterion for weak mixing}

An important ingredient to handle the case where $t \notin \Z$ in the second and the third steps detailed above is provided by the so-called \textit{Veech criterion} for weak mixing, whose statement we now recall.

\begin{thm}[Veech, \cite {V4}, \textsection 7]
\label{veechcriterion}
For any Rauzy class $\mathfrak{R}\subset \mathfrak{S}_d^0$ there exists an open set $U_{\mathfrak{R}} \subset \mathbb{P}^{d-1}_+\times \mathfrak{R}$ with the following property.  Assume that the orbit of $([\lambda],\pi) \in \mathbb{P}^{d-1}_+\times \mathfrak{R}$
under  the Rauzy induction map $\mathcal{R}_R$ visits $U_{\mathfrak{R}} $ infinitely many times. If there exists a non-constant measurable solution $\phi\colon I \to \C$ to the equation

\begin{equation}\label {phih}
\phi(f(x))=e^{2 \pi \mathrm{i} t h_j} \phi(x)\, ,  \quad x\in I_{j}^\lambda,
\end{equation}
with $t \in \R$, $h \in \R^d$, then
\begin{equation*}
\lim_{\ntop{n \to \infty} {\mathcal{R}_R^n([\lambda],\pi) \in U_{\mathfrak{R}}}}\,
\|B_n^R([\lambda],\pi) \cdot t h\|_{\R^d/\Z^d}=0.
\end{equation*}

\end{thm}

In the context of weak mixing of i.e.t.'s, it is especially interesting to apply Veech criterion in the particular case where $h=(1,\dots,1)$ (see Equation (\ref{equcontspec})).

\subsubsection{Analysis of the case where ${(1,\dots,1) \notin H(\pi)}$}\label{subs11}

Let us recall a theorem due to Veech and which states almost sure weak mixing in the case where $(1,\dots,1) \notin H(\pi)$.

\begin{thm}[Veech, \cite{V4}, \textsection 1-6]\label{thm111veech}
Let $\pi \in \mathfrak{S}_d^0$ and suppose that $(1,\dots,1) \not\in H(\pi)$. Then for almost every $\lambda \in \R_+^d$, $f(\lambda,\pi)$ is weakly mixing.
\end{thm}

We want to improve this result in the following way.

\begin{theo}\label{thm111bis}
Let $\pi \in \mathfrak{S}_d^0$ and assume that $(1,\dots,1) \notin H(\pi)$. Then the set of $[\lambda] \in \mathbb{P}_+^{d-1}$ such that $f([\lambda],\pi)$ is not weakly mixing has Hausdorff dimension strictly less than $d-1$.
\end{theo}

\noindent The proof follows \cite{V4}. Fix a Rauzy class $\mathfrak{R} \subset \mathfrak{S}_d^0$, let $\pi \in \mathfrak{R}$ and let $[\lambda] \in \mathbb{P}_{+}^{d-1}$. 

\begin{prop}[Proposition 6.5, Veech \cite {V4}]\label{propveech}
With notations and assumptions as above, suppose that $h \in \R^d$ is such that there exists an infinite set $E$ of natural numbers satisfying
\begin{equation}\label{hypothesis}
\lim_{\ntop{n \to \infty}{n \in E}} e^{2 \pi \mathrm{i} (B_n^R([\lambda],\pi) \cdot h)_j} = 1,\quad 1 \leq j \leq d.
\end{equation}
Then $h \cdot b^s \in \Z$ for all $s \in \Sigma(\pi)$, where $b^s$ is defined as in Subsection \ref{invariantsubspace}.
\end{prop}

Let $U_\mathfrak{R} \subset \mathbb{P}_+^{d-1} \times \mathfrak{R}$ be the open set involved in the statement of Veech criterion (Theorem \ref{veechcriterion}).
Reasoning as in Subsection \ref{subsectionergodicity}, we get that for $n \geq 0$ sufficiently large, there exists a connected component $\Delta \times \{\pi\} \subset U_\mathfrak{R}$ of the domain of definition of $\mathcal{R}_R^n$ which is labeled by a path $\gamma_0 \in \Pi(\mathfrak{R})$ of length $n$ starting and ending at $\pi$ and such that the matrix $B_{\gamma_0}$ has all its entries positive. In particular, the corresponding simplex $\Delta=B_{\gamma_0}^* \cdot \mathbb{P}_{+}^{d-1}$ is compactly contained in $\mathbb{P}_+^{d-1}$. 
By expansiveness of Rauzy renormalization we know that it is sufficient to estimate the dimension of non-recurrent parameters which belong to $\Delta$. As in Subsection \ref{rauzyfast}, let $T\colon \Delta \to \Delta$
be the map induced by $\mathcal{R}_R$ on $\Delta$; we have seen that it has bounded distortion and fast decay. By Proposition \ref{escapingpoints}, we deduce that the set $\Delta \backslash \Delta^\infty$ of $[\lambda] \in \Delta$ such that the orbit of $([\lambda],\pi)$ under $\mathcal{R}_R$ visits $\Delta\times \{\pi\}$ finitely many times has Hausdorff dimension strictly less than $d-1$. Veech criterion together with Proposition \ref{propveech} thus imply:

\begin{thm}\label{theorem63veech}
Let $\pi \in \mathfrak{S}_d^0$; for every $[\lambda]$ in a set whose complement in $\mathbb{P}_+^{d-1}$ has Hausdorff dimension strictly less than $d-1$, and for all $h \in \R^d$, if $f=f([\lambda],\pi)$ satisfies
\begin{equation}\label{equation62veech}
\phi(f(x))=e^{2\pi \mathrm{i} h_j} \phi(x),\quad x \in I_j^{\lambda},
\end{equation}
for some non-constant measurable function $\phi\colon I \to \mathbb{C}$, then $h \cdot b^s \in \Z$ for all $s \in \Sigma(\pi)$.
\end{thm}

The analogous statement was a key ingredient in the proof by Veech of Theorem \ref{thm111veech}. 

\noindent \textit{Proof of Theorem \ref{thm111bis}.} 
Take $\pi \in \mathfrak{S}_d^0$, and assume that it satisfies $(1,\dots,1)\not\in H(\pi)$. By definition, $H(\pi)=\mathrm{Span}(\Upsilon(\pi))^\perp$, hence for some $s \in \Sigma(\pi)$, we have $(1,\dots,1) \cdot b^s \in \mathbb{Z} \backslash\{0\}$ (recall that $b^s$ has integral coordinates). 
Then, Lemma \ref{easylemma} tells us that in fact, $(1,\dots,1) \cdot b^s=\pm 1$.

Let $[\lambda]$ belong to the subset of $\mathbb{P}_+^{d-1}$ for which the conclusion of Theorem \ref{theorem63veech} holds. We know that its complement does not have full Hausdorff dimension. If $f=f([\lambda],\pi)$ is not weakly mixing, then either $f$ is not ergodic or else there exists some nonintegral $t \in \R$ and a nontrivial measurable solution to Equation (\ref{equation62veech}) for $f$ and $h:= (t,\dots,t)$. The first case was handled in Subsection \ref{subsectionergodicity}. Assume now that we are in the second case. By choosing $s$ as previously, Theorem \ref{theorem63veech} implies that $t (1,\dots,1) \cdot b^s=h \cdot b^s \in \Z$. But we also have $(1,\dots,1) \cdot b^s=\pm 1$, and by definition, $t \in \mathbb{R}\backslash \mathbb{Z}$, a contradiction.\qed\\

The case where $g=1$ is contained in the preceding theorem. Indeed, it is easy to see that under the hypothesis that $\pi$ is not a rotation, our study amounts to considering the case where $\pi$ is as \text{follows:}
\begin{equation*}
\pi(1) = 3,\ \pi(2)=2,\ \pi(3)=1,
\end{equation*}
that is $\pi=\pi_3$, where $\pi_d(j):=d+1-j,\ 1\leq j \leq d$. But for $d$ odd, we check that $(1,\dots,1) \notin H(\pi_d)$.

Theorem \ref{thm111bis} enables us to improve a previous result due to Katok and Stepin in \cite{KS1}. Here is the new statement.

\begin{theo}
If $g=1$ and $\pi$ is an irreducible permutation which is not a rotation, then $f([\lambda],\pi)$ is weakly mixing except for a set of $[\lambda]$ of Hausdorff dimension strictly less than $d-1$.
\end{theo}

\subsubsection{Analysis of the case where ${g>1}$ and ${(1,\dots,1) \in H(\pi)}$}\label{genresup1}

To prove Theorem \ref{theoprincipal}, it remains to deal with the case where $g>1$ and $(1,\dots,1) \in H(\pi)$. The reason why we wanted to get rid of the genus one case is the following: the restriction of the Rauzy cocycle to $H(\pi)$ has $g$ positive Lyapunov exponents; in particular we have at least two positive exponents when $g > 1$. As explained in the outline, this fact is central in the argument developed in \cite{AF1} and that we are going to adapt in what follows. By Veech criterion, the result we want to prove is implied by the following \text{theorem:}

\begin{thm} \label{thmA}

Let $\mathfrak{R}\subset \mathfrak{S}_d^0$ be a Rauzy class with $g>1$, let $\pi \in \mathfrak{R}$ and let $h \in H(\pi)$. Let $U \subset \mathbb{P}^{d-1}_+\times \mathfrak{R}$ be any open set. Then the set of $[\lambda] \in \mathbb{P}^{d-1}_+$ satisfying
\begin{equation}\label{equlambdath}
\limsup_{\ntop {n \to \infty} {\mathcal{R}_R^n([\lambda],\pi) \in U}}
\|B_n^R([\lambda],\pi) \cdot t h\|_{\R^d/\Z^d}=0
\end{equation}
for some $t \in \R$ such that $t h \not\in \Z^d$ has Hausdorff dimension strictly less than $d-1$.

\end{thm}

For the problem of weak mixing of i.e.t.'s, and as we have already said, we are especially interested in applying this result to $h=(1,\dots,1)$.
Moreover, the above statement deals with $th \notin \Z^d$; indeed, the case where $(t,\dots,t) \in \Z^d$ is related to Equation (\ref{equcontspec}) with $t$ an integer, and this part of the study was carried out in Subsection \ref{subsectionergodicity}.

\begin{proof}
Fix $\pi \in \mathfrak{R}$ and $h \in H(\pi)$. Let $U \subset \mathbb{P}^{d-1}_+\times \mathfrak{R}$ be any open set and assume that $[\lambda] \in \mathbb{P}^{d-1}_+$ satisfies (\ref{equlambdath}) for some $t \in \R$ such that $t h \not\in \Z^d$. We will reduce the statement to an analogous one, but for a cocycle with better properties. 

We may assume that $U$ intersects $\mathbb{P}^{d-1}_+ \times \{\pi\}$. As in Subsection \ref{subs11}, we take a connected component $\Delta \times \{\pi\} \subset U$ of the domain of $\mathcal{R}_R^n$, $n \gg 1$, such that $\Delta \Subset \mathbb{P}_+^{d-1}$. We have seen that the set of points whose orbit under $\mathcal{R}_R$ does not visit $\Delta \times \{\pi\}$ infinitely many times does not have full Hausdorff dimension so that we can restrict ourselves to points in the complement of this set. 
Let $T\colon \Delta \to \Delta$ be the map induced by $\mathcal{R}_R$ on $\Delta$.
From Lemma \ref{boundddd} we know that $T$ has bounded distortion, which will be crucial for the following.

Let $\Delta^\infty$ denote the subset of $[\lambda] \in \Delta$ to which we can apply $T$ infinitely many times.\footnote{Its complement does not have full Hausdorff dimension by Proposition \ref{escapingpoints}.} For every $[\lambda] \in \Delta^\infty$, let
\begin{equation*}
A([\lambda]):=\restriction{B_{r([\lambda])}^R([\lambda],\pi)}{H(\pi)},
\end{equation*}
where $r([\lambda])$ denotes the first-return time of $([\lambda],\pi)$ to $\Delta \times \{\pi\}$ under $\mathcal{R}_R$. Lemma \ref{fastdddd} tells us that the cocycle $(T,A)$ is fast decaying.
Moreover, $(T,A)$ is locally
constant, integral and log-integrable (the last point follows from fast decay of $A$ as was remarked in Subsection \ref{fastdecaydefin}). 
We also see that $\Theta:=\overline {\mathbb{P}^{d-1}_+}$ is \textit{adapted} to $(T,A)$, that is $A^{(l)} \cdot \Theta \subset \Theta$ for all $l$ and for every $[\lambda]$, we have
\begin{equation}\label{3.2}
    \|A([\lambda]) \cdot w\| \geq \|w\|,
\end{equation}
and
\begin{equation}\label{3.3}
    \|A_n([\lambda]) \cdot w \| \to \infty
\end{equation}
whenever $w \in \R_+^d \backslash \{0\}$ projectivizes to an element of $\Theta$. Indeed, $w\in \R_+^d\backslash \{0\}$, and the coefficients of the matrices $(A_n([\lambda]))_n$ are positive and go to infinity (see \cite{MMY}, \textsection 1.2.3-1.2.4). 

In terms of the cocycle $(T,A)$, Equation (\ref{equlambdath}) can be rewritten:
\begin{equation}\label{equcontrad}
\lim_{n \to \infty} \|A_n
([\lambda]) \cdot th\|_{\R^d/\Z^d}=0, \quad \text {for
some } t \in \R \text { such that } t h \notin \Z^d.
\end{equation}
Let us denote by $J_h$ the line spanned by $h$. In particular, for each ``bad" parameter $[\lambda]$ we have $J_h \cap (W^s([\lambda])\backslash \Z^d) \neq \emptyset$, where $$W^s([\lambda]):=\{w \in \mathbb{R}^d,\ \|A_n([\lambda]) \cdot w \|_{\R^d/\Z^d} \to 0\}$$ denotes the weak-stable space for $\{A_n([\lambda])\}_{n \geq 0}$.
Let then
\begin{equation*}
\mathcal{E}_{t,h} := \{[\lambda] \in \Delta,\ th \in W^s([\lambda])\}.
\end{equation*}

\noindent In Section \ref{proofoftheomaintheo} we exhibit a set $ \mathcal{E}_h$ depending only on $h$ (in fact on the line $J_h$) and such that $\bigcup\limits_{t\in \R,\ th \not\in \Z^d} \mathcal{E}_{t,h} \subset \mathcal{E}_h$. Theorem \ref{thmA} follows from the next result, whose proof is deferred to Subsection \ref{proofoftheomaintheo}.

\begin{theo}\label{maintheo}
We have $\mathrm{HD}(\mathcal{E}_h) < d-1.$
\end{theo}
\end{proof}

\subsection{Proof of Theorem \ref{theoprincipal}}

Using the previous results, we give here the proof of Theorem \ref{theoprincipal}. Let $\mathfrak{R} \subset \mathfrak{S}_d^0$ and take $\pi\in\mathfrak{R}$ an irreducible permutation which is not a rotation. Assume $[\lambda] \in \mathbb{P}_+^{d-1}$ is such that $f=f([\lambda],\pi)$ is not weakly mixing. If $(1,\dots,1) \not\in H(\pi)$, then Theorem \ref{thm111bis} gives the result. Let us then consider the case where $g > 1$ and $(1,\dots,1) \in H(\pi)$. As we have said, the fact that $f$ is not weakly mixing means that either it is not ergodic or that for some $t \in \R \backslash \Z$, there is a nonzero measurable solution $\phi\colon I \to \C$ to 
\begin{equation}\label{eqfff}
\phi \circ f=e^{2 \pi \mathrm{i} t} \phi.
\end{equation} 
By Theorem \ref{theouniq}, we know that $f([\lambda],\pi)$ is uniquely ergodic, hence ergodic, but for a set of $[\lambda]$ with non-full Hausdorff dimension. It thus remain to deal with Equation (\ref{eqfff}) for $t \not\in \Z$. 

Assume (\ref{eqfff}) holds for some nonzero $\phi$ and $t \not\in \Z$. Let $U_\mathfrak{R}$ be the open set given by Veech criterion (Theorem \ref{veechcriterion}), that we apply here with this choice of $t$ and $h=(1,\dots,1) \in H(\pi)$. We take a connected component $\Delta \times \{\pi\} \subset U_\mathfrak{R}$ of the domain of $\mathcal{R}_R^n$, $n \gg 1$, such that $\Delta \Subset \mathbb{P}_+^{d-1}$, and we consider the cocycle $(T,A)$ where $T\colon \Delta\to \Delta$ is the map induced by Rauzy renormalization, and $A$ is induced by the Rauzy cocycle $\restriction{B^R}{H(\pi)}$. We know from Theorem \ref{escapingpoints} that we can restrict ourselves to the set $\Delta^\infty$ of $[\lambda]$ to which we can apply $T$ infinitely many times since the set of escaping points does not have full Hausdorff dimension. By Veech criterion, (\ref{eqfff}) implies that
\begin{equation*}
\lim_{n \to \infty}
\|A_n([\lambda]) \cdot t h\|_{\R^d/\Z^d}=0,
\end{equation*}
hence $[\lambda] \in \mathcal{E}_{t,h}$ with our previous notations. We deduce that $[\lambda] \in \mathcal{E}_h=\mathcal{E}_{(1,\dots,1)}$, but by Theorem \ref{maintheo}, the latter does not have full Hausdorff dimension, which concludes. 

\subsection{Weak mixing for flows}

\subsubsection{Special flows}

Any translation flow on a translation surface can be regarded, by considering its return map to a transverse interval, as a {\it special flow} (suspension flow) over some interval exchange
transformation with a roof function constant on each subinterval.

Let $F=F(h,\lambda,\pi)$ be the special flow over the i.e.t.\  $f=f(\lambda,\pi)$ with roof function
specified by the vector $h\in \R^d_+$, that is, the roof function is constant equal to $h_{i}$ on the
subinterval $I_{i}= I_{i}^\lambda$, for all $i\in \{1,\dots,d \}$. It is possible to see that if $F$ is a translation flow then necessarily $h\in H(\pi)$. As for interval exchange transformations, we can projectivize the length data $\lambda$.

As explained earlier, the phase space of $F$  is the union of disjoint rectangles $D:=\bigcup\limits_{i} I_i \times [0,h_i)$, and the flow $F$
is completely determined by the conditions $F_s(x,0)=(x,s)$, $x \in I_i$, $s<h_i$, $F_{h_i}(x,0) =(f(x),0)$,
for all $i\in \{1,\dots,d\}$. Weak mixing for the flow $F$ is equivalent to the non-existence of non-constant measurable solutions $\phi\colon D \to \C$ to
\begin{equation*}
\phi(F_s(x))=e^{2 \pi \mathrm{i} t s} \phi(x),
\end{equation*}
for any $t \in \R$, or,  in terms of the i.e.t.\ $f$,
\begin{enumerate}
\item $f$ is ergodic;
\item  for any $t \neq 0$ there are no nonzero measurable solutions $\phi\colon I \to \C$ to Equation (\ref {phih}).
\end{enumerate}

\begin{thm}\label{theeeo}

Let $\pi \in \mathfrak{S}_d^0$ with $g>1$.
For every $h \in (H(\pi) \cap \R^d_+)$ and every $[\lambda] \in \mathbb{P}_+^{d-1}$ except for a subset of Hausdorff dimension strictly less than $d-1$,
the special flow $F(h,\lambda,\pi)$ is weakly mixing.

\end{thm}
\noindent The proof is an immediate consequence of Veech criterion and of Theorem \ref{thmA}.

\subsubsection{Weak mixing for translation flows}\label{sect flows trans}

Let $\mathfrak{R} \subset \mathfrak{S}_d^0$ be a Rauzy class with $d > 1$, and let $\pi \in \mathfrak{R}$. We have seen that the associate space of translation flows is parametrized by $(h,[\lambda]) \in H(\pi) \times \mathbb{P}_+^{d-1}$, where $H(\pi)$ has dimension $2g$. As for i.e.t.'s, we want to prove that the set of $(h,[\lambda]) \in H(\pi) \times \mathbb{P}_+^{d-1}$ such that the corresponding flow is not weakly mixing has Hausdorff dimension strictly less than $2g + d -1$. If we look at Theorem \ref{theeeo}, the point is that the Hausdorff dimension of a product is not necessarily less than the sum of the Hausdorff dimension of the two factors.\footnote{For instance, it is possible to find $X$ and $Y$, $\mathrm{HD}(X)=\mathrm{HD}(Y)=0$, but $\mathrm{HD}(X \times Y)=1$.} Considering the return map to a transverse interval, and by Veech criterion, it is enough to show the following.

\begin{theo}
Let $U \subset \mathbb{P}^{d-1}_+\times \mathfrak{R}$ be any open set. The set of $(h,[\lambda]) \in (H(\pi)\backslash \mathbb{Z}^d) \times \mathbb{P}_+^{d-1}$ satisfying
\begin{equation*}
\limsup_{\ntop {n \to \infty} {\mathcal{R}_R^n([\lambda],\pi) \in U}}
\|B_n^R([\lambda],\pi) \cdot h\|_{\R^d/\Z^d}=0
\end{equation*}
has Hausdorff dimension strictly less than $2g+d-1$.
\end{theo}

Again we restrict ourselves to a simplex $\Delta\Subset \mathbb{P}_+^{d-1}$, and we define a cocycle $(T,A)$, where $T\colon \Delta \to \Delta$ is the map induced by Rauzy renormalization. We let 
\begin{equation*}
\mathcal{F}:=\{(h,[\lambda]) \in H(\pi) \times \Delta,\ h \in W^s([\lambda])\backslash \mathbb{Z}^d\}.
\end{equation*}

\noindent It is then sufficient to show the following result, which is proved in Subsection \ref{proofweakmixtransl}. 
\begin{theo}\label{wmixt}
We have $\mathrm{HD}(\mathcal{F})< 2g+d-1$.
\end{theo}

\section{Study of the weak-stable lamination}\label{sectionweakstable}

\subsection{The probabilistic argument of Avila and Forni \cite{AF1}}\label{subsproba}

As in \cite{AF1}, we consider here an abstract locally constant integral cocycle $(T,A)$. We also assume that $T$ has bounded distortion\footnote{with respect to a measure $\mu$.}, and that $(T,A)$ is fast decaying; as we have already seen, this implies that $(T,A)$ is log-integrable. Let $\Theta \subset \mathbb{P}_{+}^{d-1}$ be a compact set \textit{adapted} to $(T,A)$ (see (\ref{3.2}) and (\ref{3.3}) for the definition). We define $\mathcal{J}=\mathcal{J}(\Theta)$ to be the set of lines in $\R^d$, parallel to some element of $\Theta$ and not passing through $0$.

Recall the definition of the \textit{weak-stable lamination} associated to the cocycle $(T,A)$ at some point $x \in \Delta\text{:}$
\begin{equation*}
W^s(x):=\{w \in \mathbb{R}^d,\ \|A_n(x) \cdot w \|_{\R^d/\Z^d} \to 0\},
\end{equation*}
where $\|\cdot\|_{\R^d/\Z^d}$ is the euclidean distance from $\Z^d$.\comm{For almost every $x \in \Delta$, the space $W^s(x)$ is a union of translates of $E^s(x)$. In the case where the cocycle is bounded, namely if $A : \Delta \to \mathrm{GL}(d,\R)$ is essentially bounded, we can see that $W^s(x)=\bigcup\limits_{c \in \Z^d} E^s(x) + c$. However, most of the time $W^s(x)$ appears to have a more complicated structure ; in particular, it may be the union of uncountably many translates of $E^s(x)$.
Veech criterion tells us that the problem of weak mixing can essentially be reduced to proving some exclusion result for the weak-stable lamination of the cocycle introduced in Section \ref{mainresultsection}. More precisely, given $(t,h) \in \R\times H(\pi)$ such that $th \not \in \Z^d$, we want the set of parameters $[\lambda]$ for which $th \in W^s([\lambda])$ to be ``small" (in terms of Hausdorff dimension). This is the reason why it is crucial for our study to have a better understanding of the weak-stable space.\\} For $0 < \delta < 1/10$ and $n \geq 0$, we define
\begin{equation*}
W_{\delta,n}^s(x):=\{w \in B_\delta(0),\ \forall k \leq n,\ \|A_k(x) \cdot w\|_{\R^d/\Z^d} < \delta\}
\end{equation*}
and we let
\begin{equation*}
W_\delta^s(x):=\bigcap\limits_{n \geq 0} W_{\delta,n}^s(x).
\end{equation*}
For every such $\delta$ and for all $w \in W^s(x)$, we see that there exists $n_0 \geq 0$ such that for every $n \geq n_0$, there exists $c_n(w) \in \Z^d$ with $A_n(x) \cdot w -c_n(w) \in W_\delta^s(T^{n}(x))$.\\

Recall that $\Omega$ denotes the set of all finite words with integer entries.
For any $0<\delta < 1/10$ and $\underline{l} \in \Omega$, let $\phi_\delta(\underline{l},J)$ be the number of connected components of the set $A^{\underline{l}}(J \cap B_\delta(0)) \cap B_\delta(\Z^d\backslash\{0\})$. If $J \in \mathcal{J}$, we let $\|J\|$ denote the distance between $J$ and $0$. Given $J$ with $\|J\| < \delta$ and $\underline{l} \in \Omega$, let $J_{\underline{l},1},\dots,J_{\underline{l},\phi_\delta(\underline{l},J)}$ be all the lines of the form $A^{\underline{l}} \cdot J - c$ where $A^{\underline{l}}\cdot(J \cap B_\delta (0)) \cap B_\delta(c) \neq \emptyset$ with $c \in \Z^d\backslash\{0\}$: such lines are called \textit{non-trivial children} of $J$. We also define $J_{\underline{l},0}:=A^{\underline{l}} \cdot J$. If $\|A^{\underline{l}} \cdot J\|< \delta$, then $J_{\underline{l},0}$ is called a \textit{trivial child} of $J$.  

\begin{rem}\label{remnontrivchildren}
We define $\phi_\delta(\underline{l}):=\sup\limits_{J\in \mathcal{J}}\phi_\delta(\underline{l},J)$. It is clear that by making $\delta \to 0$, non-trivial children become rarer, which means that for any (fixed) $\underline{l} \in \Omega$, the function $\delta \mapsto \phi_\delta(\underline{l})$ is non-decreasing. We see that there exists some $\delta_{\underline{l}} > 0$ such that for $\delta < \delta_{\underline{l}}$, we have
$\phi_\delta(\underline{l})=0$.
\end{rem}

\begin{figure}[H]
\begin{center}
    \includegraphics [width=8cm]{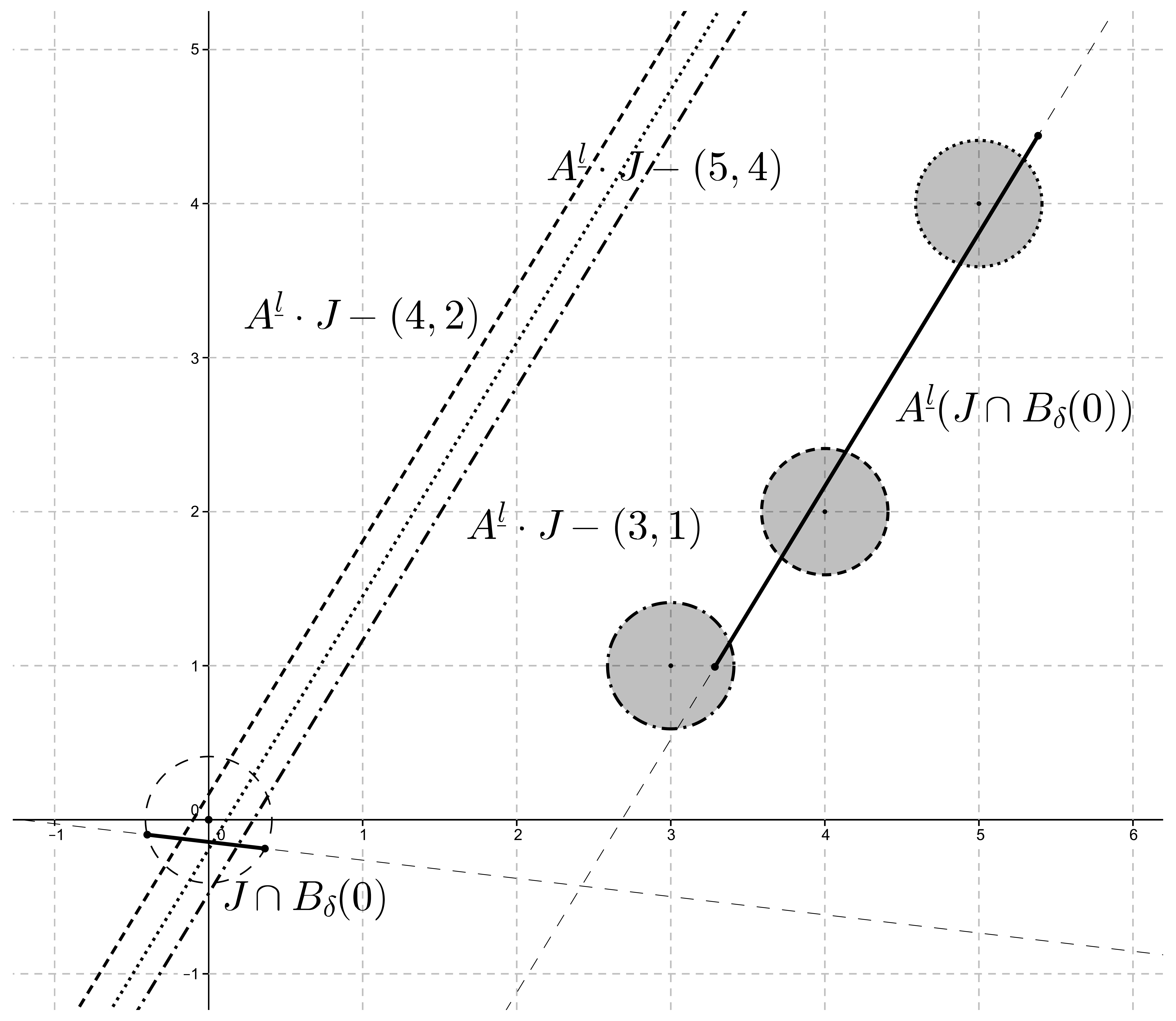}
    \caption{\label{Child}A line and its first-generation children}
\end{center}
\end{figure}

Fix $N \in \N\backslash\{0\}$, and let $\Omega^{N}$ be the set of all words of length $N$.
The set $\Delta$ admits a countable partition $\bigcup\limits_{l \in \Z} \Delta^{(l)}$, and since $A$ is locally constant, it is possible to associate to each $x \in \Delta$ a word $\underline{l_0}(x)\underline{l_1}(x)\underline{l_2}(x)\dots$, where $\underline{l_m}(x) \in \Omega^N$ is such that $T^{mN}(x) \in \Delta^{\underline{l_m}(x)}$. We call such a $\underline{l_m}(x)$ a ``slice" of $x$. To each slice $\underline{l_m}(x)$ of $x$ corresponds the matrix $A^{\underline{l_m}(x)}=A_N(T^{mN}(x))$. The integer $N$ enables us to accelerate the process so as to ``see'' the Lyapunov exponents.\\

In order to study the weak-stable lamination, Forni and the first author introduced the following process. Given $x \in \Delta$, $\delta > 0$ and a line $J$ with $\|J\| < \delta$, we apply successively the matrices $(A^{\underline{l_m}(x)})_m$ to $J$ and estimate the number of children at each step. \textit{A priori} this number might grow exponentially fast; in fact the probability that this process survives up to time $m$ goes to zero exponentially fast with $m$ as we will see.\\

Given $0 < \eta < 1/10$, select a finite set $Z \subset \Omega^N$ of big measure:
\begin{equation}\label{aux}
\mu\left(\bigcup\limits_{\underline{l} \in Z} \Delta^{\underline{l}}\right) > 1 - \eta.
\end{equation}
\indent Since the cocycle is locally constant and log-integrable, there exists $0 < \eta_0 < 1/10$ such that for $0<\eta < \eta_0$, 
\begin{equation}\label{3.13}
    \sum\limits_{\underline{l} \in \Omega^N\backslash Z}\ln \|A^{\underline{l}}\|_0 \mu(\Delta^{\underline{l}}) < \frac{1}{2},
\end{equation}
which will ensure that in average, matrices with label in $Z$ overcome the others.

Moreover, since $Z$ is finite, it is possible by Remark \ref{remnontrivchildren} to ensure that for each $J \in \mathcal{J}$, when the matrix we apply to $J$ is labeled by a word in $Z$, then the only potential child is trivial. But the existence of at least two positive Lyapunov exponents generates a drift that makes each child further from the origin than the original line was. Therefore, if we can do so that a typical $x \in \Delta$ has enough slices in $Z$, then the previous process has good chances to have finite life expectancy.

For every $m \geq 0$, we define 
$$\Gamma_\delta^m(J):=\{x \in \Delta,\ J \cap W_{\delta,mN}^s(x) \neq \emptyset\}.
$$ 
It is the set of points for which the above process will last a minimum of $m$ steps. The main result of this section is the following.
\begin{prop}\label{mainresultprob}
There exists $N_0 \in \mathbb{N}\backslash\{0\}$ such that for $N>N_0$, we may find constants $C >0$, $\kappa > 0$ and $\rho >0$ such that for every $J \in \mathcal{J}$, 
\begin{equation*}
\mu(\Gamma_\delta^m(J)) \leq C e^{-\kappa m} \|J\|^{-\rho}.
\end{equation*}
\end{prop}

Note that the dependence in $J$  in the previous expression is quite reasonable: the process will survive longer if the initial line $J$ is chosen close to the origin, that is, when $\|J\|$ is small.\\

To prove Proposition \ref{mainresultprob}, we condition the probabilities according to whether we are in $Z$ or not: given measurable sets $X,Y \subset \Delta$ such that $\mu(Y) >0$, and with the notations of Section \ref{background}, let
\begin{equation}\label{3.11}
    P_{\nu}(X|Y) := \frac{\nu(X \cap Y)}{\nu(Y)},\quad \nu \in \mathcal{M}:=\{\mu^{\underline{l}},\ \underline{l} \in \Omega^N\},
\end{equation}
\begin{equation}\label{3.12}
    P(X|Y):=\sup\limits_{\nu \in \mathcal{M}}P_\nu(X|Y).
\end{equation}

We introduce functions $\psi,\Psi$ to encode whether slices belong to $Z$ or not. More precisely, let $\psi\colon \Omega^N \to \Z$ be such that $\psi(\underline{l})=0$ if $\underline{l} \in Z$, and $\psi(\underline{l})\neq\psi(\underline{l}')$ whenever $\underline{l} \neq \underline{l}'$ and $\underline{l}, \underline{l}' \not\in Z$. We denote by $\hat{\Omega}^N$ the set of all words whose length is a multiple of $N$. Let then $\Psi\colon \hat\Omega^N \to \Omega$ be given by $\Psi(\underline{l}^{1}\dots\underline{l}^{m}) = \psi(\underline{l}^{1})\dots\psi(\underline{l}^{m})$, where $\underline{l}^{i}\in\Omega^N$, $i=1,\dots,m$. For any $\underline{d} \in \Omega$, define $$\hat\Delta^{\underline{d}}:=\bigcup\limits_{\underline{l} \in \Psi^{-1}(\underline{d})}\Delta^{\underline{l}}.$$

Let $\rho > 0$. For each $\underline{d} \in \Omega$, $|\underline{d}|=m$, we define $C(\underline{d}) \geq 0$ as the smallest number such that
\begin{equation}\label{3.26}
    P(\Gamma_\delta^m(J)|\hat\Delta^{\underline{d}}):=\sup\limits_{\nu \in \mathcal M}P_\nu(\Gamma_\delta^m(J)|\hat\Delta^{\underline{d}}) \leq C(\underline{d}) \|J\|^{-\rho},\quad J \in \mathcal{J}.
\end{equation}

Note that $C(\underline{d}) \leq 1$ for all $\underline{d}$; indeed, if $\|J\|>\delta$, then $\Gamma_\delta^m(J) = \emptyset$, else $P(\Gamma_\delta^m(J)|\hat\Delta^{\underline{d}}) \|J\|^{\rho} \leq \delta^\rho \leq 1$. The following technical estimate is the key step in the proof of Proposition \ref{mainresultprob}.

\begin{lem}[Claim 3.7, Avila-Forni \cite{AF1}]\label{keylemma} There exists $N_0\in \mathbb{N}\backslash\{0\}$ such that for $N>N_0$, if $Z \subset \Omega^N$ and $0 < \eta < \eta_0$ are taken such that \eqref{aux} and \eqref{3.13} hold, then there exists $\rho_0(Z) > 0$ such that for $0<\rho< \rho_0(Z)$, and for $0<\delta < 1/10$ sufficiently small, we have for any $\underline{d}=(d_1\dots d_m)\in \Omega$:
\begin{equation}\label{3.27}
    C(\underline{d}) \leq \prod\limits_{d_i=0}(1-\rho) \prod\limits_{d_i \neq 0,\ \psi(\underline{l}^{i})=d_i} \|A^{\underline{l}^i}\|_0^\rho (1+\|A^{\underline{l}^i}\|_0(2\delta)^\rho).
\end{equation}
\end{lem}

This means that at each occurrence of a slice in $Z$, the previous quantity decreases by at least a given factor. Therefore, if the behavior of matrices with label in $Z$ is prevalent, it follows that for a typical infinite word $\underline{d}=d_1 d_2 d_3\dots$, the probability $P(\Gamma_\delta^m(J)|\hat\Delta^{\underline{d}^m})$ goes to zero exponentially fast with respect to $m$, where we have set $\underline{d}^m:=d_1 d_2 \dots d_m$.

\subsection{Large deviations: proof of Proposition \ref{mainresultprob}}

At this point we fix $N>N_0$, $Z \subset \Omega^N$, $0 < \eta < \eta_0$, $0< \rho < \rho_0(Z)$ and $0<\delta<1/10$ in such a way that the assumptions of Lemma \ref{keylemma} are satisfied; we also assume that $\delta$ is small enough such that\footnote{By \eqref{aux} and \eqref{3.13}, the left hand side is less than $\frac{1}{2} \rho(\eta-\frac{1}{2})<0$ when $\delta \to 0$.}
\begin{equation}\label{3.24}
    \sum\limits_{\underline{l} \in \Omega^N\backslash Z}(\rho \ln \|A^{\underline{l}}\|_0 + \ln(1+\|A^{\underline{l}}\|_0(2\delta)^\rho))\mu(\Delta^{\underline{l}})-\rho \mu\left(\bigcup\limits_{\underline{l} \in Z} \Delta^{\underline{l}}\right)=\alpha <0.
\end{equation}
This ensures that on average, the behavior of words in $Z$ prevails. For any $x \in \Delta$, we let
\begin{equation*}
\gamma(x) := \left\{
    \begin{array}{ll}
        -\rho & \mbox{if}\ x \in \bigcup\limits_{\underline{l} \in Z} \Delta^{\underline{l}}, \\
        \rho \ln \|A^{\underline{l}}\|_0+\ln(1+\|A^{\underline{l}}\|_0(2 \delta)^\rho) & \mbox{if}\ x \in \Delta^{\underline{l}},\ {\underline l}\in \Omega^N\backslash Z.
    \end{array}
\right.
\end{equation*}

For each $k \in \N$, we define the random variable
\begin{displaymath}
X_k\colon
\left\{
  \begin{array}{rcl}
    \Delta & \longrightarrow &\mathbb{R} \\
    x & \longmapsto & \gamma(T^{kN}(x)),
  \end{array}
\right.
\end{displaymath}
and we set $S_m:= \sum\limits_{k=0}^{m-1} X_k$. From (\ref{3.24}) and the $T-$invariance of $\mu$, we get: for every $k \geq 0$,
\begin{equation*}
\mathbb{E}(X_k):=\int_\Delta \gamma(T^{kN}(x))\ d\mu(x)=\alpha < 0.
\end{equation*}

For $x \in \hat{\Delta}^{\underline{d}}$, $|\underline{d}|=m$, let $C_m(x):=C(\underline{d})$. Then by (\ref{3.27}),
\begin{equation*}
    \ln C_m(x) \leq S_m(x).
\end{equation*}
Since $T$ is ergodic, Birkhoff's ergodic theorem implies that for almost every $x \in \Delta$,
\begin{equation*}
\limsup \limits_{m \to \infty}\frac{\ln C_m(x)}{m} \leq \limsup \limits_{m \to \infty}\frac{S_m(x)}{m} = \alpha<0.
\end{equation*}
It follows that for almost every $x \in \Delta$, there exists $m_0$ such that for $m \geq m_0$ we have $C_m(x) < e^{\frac{m \alpha}{2}}$. Define
\begin{equation*}
E_m:=\{x \in \Delta,\ C_m(x) \geq e^{\frac{m\alpha}{2}}\}.
\end{equation*}
We show the following large deviations result for $C_m$:
\begin{lemma}\label{lemmalargedev}
There exist $\beta < 1$ and $\widetilde C > 0$ such that $\mu(E_m) \leq \widetilde C \beta^m$.
\end{lemma}

To prove this result, we will use a general estimate obtained in \cite{AD} and that we now recall. Let $U\colon \Delta \to \Delta$ be a transformation with bounded distortion with respect to the reference measure $\mu$, let $\nu$ be the invariant measure, and $(U,B)$ a locally constant cocycle over $U$. The {\it expansion constant} of $(U,B)$ is the maximal $c \in \R$ such that for all $v\neq 0$ and almost every $x \in \Delta$,\footnote{The limit exists by Oseledets theorem applied to $\nu$.}
\begin{equation*}
\lim\limits_{n \to \infty} \frac{1}{n}\ln\|B_n(x) \cdot v\| \geq c.
\end{equation*}

The following theorem tells us that the measure of points which exhibit anomalous Lyapunov behavior at time $n$ decays exponentially fast with $n$.

\begin{thm}[Theorem 25, Avila-Delecroix \cite{AD}] \label{thm:LD_for_expansion}
Assume that $B$ is fast decaying. For every $c'<c$,
there exist $C_3>0$, $\alpha_3>0$ such that for every unit vector $v$,
\begin{equation*}
\mu \{x,\ \|B_n(x) \cdot v\| \leq e^{c' n}\} \leq C_3 e^{-\alpha_3 n}.
\end{equation*}
\end{thm}

\noindent \textit{Proof of Lemma \ref{lemmalargedev}.} With the notations introduced before, we define the cocycle $(U,B)$, where $U:=T^N$ and $B(x):=e^{-\gamma(x)} \in \mathrm{GL}(1,\R)$. When $x \in \Delta^{\underline{l}},\ |{\underline{l}}|=N$, we denote $B^{({\underline{l}})} := B(x)$. Note that the map $U$ has bounded distortion.

Since $Z$ is finite, for $n$ big enough, $|B^{({\underline{l}})}|_0 \geq n$ implies that ${\underline l}\in \Omega^N \backslash Z$, and it follows from the definition of $\gamma$ that  
\begin{equation*}
|B^{({\underline{l}})}|_0 =\|A^{{\underline{l}}}\|_0^\rho (1+\|A^{{\underline{l}}}\|_0(2\delta)^{\rho}) \geq \|A^{{\underline{l}}}\|_0^\rho,
\end{equation*}
and then $B$ is fast decaying since $A$ is.

We deduce from (\ref{3.24}) the value of the expansion constant of the cocycle $(U,B)$: we have $c=-\alpha > 0$. Indeed, $B_m(x)=e^{-S_m(x)}$, so by Birkhoff's ergodic theorem, for $|v|=1$ and almost every $x \in \Delta$,
\begin{equation*}
\lim\limits_{n \to \infty} \frac{1}{m} \ln |B_m(x) \cdot v| = \lim\limits_{m \to \infty} -\frac{S_m(x)}{m} \stackrel{(\ref{3.24})}{=} -\alpha.
\end{equation*}
We apply Theorem \ref{thm:LD_for_expansion} with $c':=-\alpha/2 < c$ to get $C_3>0$, $\alpha_3 > 0$ such that
\begin{equation*}
\mu\{x\in \Delta,\ B_m(x) \leq e^{c' m}\} \leq C_3 e^{-\alpha_3 m}.
\end{equation*}
Since $e^{S_m(x)} = (B_m(x))^{-1}$, the last inequality gives
\begin{equation*}
\mu\left(\left\{x \in \Delta,\ \frac{S_m(x)}{m} \geq \frac{\alpha}{2}\right\}\right) \leq C_3 e^{-\alpha_3 m}.
\end{equation*}
By definition, $\ln C_m(x) \leq S_m(x)$, hence
\begin{align*}
\mu(E_m) &= \mu(\{x\in \Delta,\ C_m(x) \geq e^{\frac{m\alpha}{2}}\})\\ 
&\leq \mu\left(\left\{x\in \Delta,\ \frac{S_m(x)}{m} \geq \frac{\alpha}{2}\right\}\right)\leq C_3 e^{-\alpha_3 m},
\end{align*} which concludes. \qed\\

Using Lemma \ref{lemmalargedev} and the fact that $C_m(x) \leq 1$, we therefore obtain
\begin{align*}\label{3.33}
    \int_{\Delta} C_m(x)\ d\mu(x) &= \int_{E_m}C_m(x) d\mu(x) + \int_{\Delta\backslash E_m} C_m(x) d\mu(x)\\
    &\leq \mu(E_m) + \int e^{m\alpha/2} d\mu(x)\\
    &\leq \widetilde C \beta^m+ e^{m\alpha/2}\\
    &\leq C e^{-\kappa m},
\end{align*}
with $\kappa := \min(-\ln \beta,-\alpha/2) > 0$ and $C:=1+\widetilde C >0$.\\

From (\ref{3.26}) and the formula of total probability, we thus get
\begin{align*}
\mu(\Gamma_\delta^m(J)) &\leq \sum\limits_{\underline{d} \in \hat{\Omega},\ |\underline{d}|=m} \mu(\hat{\Delta}^{\underline{d}}) P_\mu(\Gamma_\delta^m(J)|\hat{\Delta}^{\underline{d}})\\
& \leq \int_\Delta C_m(x) \|J\|^{-\rho}\ d\mu(x) \leq C e^{-\kappa m}\|J\|^{-\rho},
\end{align*}
which concludes the proof of Proposition \ref{mainresultprob}. \qed

\section{End of the proof of weak mixing}

\subsection{Weak mixing for i.e.t.'s: proof of Theorem \ref{maintheo}}\label{proofoftheomaintheo}

We conclude here the proof of Theorem \ref{maintheo}. 
We consider the cocycle $(T,A)$ defined in Subsection \ref{genresup1}. Note that it meets all the assumptions made in Section \ref{sectionweakstable}. Since $T$ is fast decaying, it is also legitimate to use the results of Subsection \ref{fastdecayhausdorffdim} to convert Proposition \ref {mainresultprob}
into an estimate on Hausdorff dimension. With the notations introduced in Subsection \ref{subsproba}, let $\Theta:=\overline{\mathbb{P}^{d-1}_+}$ and $\mathcal{J}=\mathcal{J}(\Theta)$, and choose some integer $N \geq 0$ sufficiently big so that the assumptions of Lemma \ref{keylemma} are satisfied.

Let $J$ be a line in $\mathcal{J}$. For every $m \geq 0$, we denote
\begin{equation}\label{defgamma}
\Gamma_\delta^m(J):=\{[\lambda] \in \Delta,\ J \cap W_{\delta,mN}^s([\lambda]) \neq \emptyset\}.
\end{equation}
With our previous notations, it is clear that $\Gamma_\delta^m(J)$ is a union of $\hat\Delta^{\underline{d}}$, with $|\underline{d}|=m$. Now, Proposition \ref{mainresultprob} implies that there exist constants $C,\kappa,\rho >0$ such that $\mu(\Gamma_\delta^m(J)) \leq C e^{-\kappa m}\|J\|^{-\rho}$, which yields
\begin{equation*}
\limsup\limits_{m \to \infty} -\frac{1}{m} \ln \mu(\Gamma_\delta^m(J)) \geq \kappa >0.
\end{equation*}
Recall that $T$ is fast decaying for some fast decay constant $\alpha_1>0$. By Theorem \ref{thm:HD}, we get\footnote{The sequence $(\Gamma_\delta^m(J))_m$ is decreasing, so that $\liminf\limits_{m \to \infty} \Gamma_\delta^m(J)=\bigcap\limits_{m\geq 0} \Gamma_\delta^m(J)$.}
\begin{equation}\label{esthd}
\mathrm{HD}\left(\bigcap\limits_{m\geq 0} \Gamma_\delta^m(J)\right) \leq d-1-\min(\kappa,\alpha_1).
\end{equation}

For every $n \geq 0$, the same reasoning remains true for $T^{-n}\left( \bigcap\limits_{m \geq 0} \Gamma_\delta^m(J)\right)$. Indeed, since $T$ leaves $\mu$ invariant, we get as before
\begin{equation*}
\limsup\limits_{m\to \infty}-\frac{1}{m} \ln \mu(T^{-n}\left(\Gamma_\delta^m(J)\right)) \geq \kappa,
\end{equation*}
and therefore,
\begin{equation}\label{esthd2}
\mathrm{HD}\left(\bigcap\limits_{m\geq 0}T^{-n}\left(\Gamma_\delta^m(J)\right)\right) \leq d-1-\min(\kappa,\alpha_1).
\end{equation}

As in Subsection \ref{genresup1}, we consider $h \in H(\pi)$ and $t \in \R$ satisfying $th \not \in \Z^d$. 
Let $x=[\lambda] \in \Delta$ be such that $th \in W^s(x)$, i.e., $x \in \mathcal{E}_{t,h}$. We fix some small $\delta > 0$. By definition, $th \in W^s(x)$ means that there is $n_0=n_0(x)\geq 0$ such that for every $n \geq n_0$, there exists $c_n(x) \in \Z^d$ satisfying
\begin{equation}\label{propth}
A_n(x) \cdot th - c_n(x) \in W_\delta^s(T^n(x)).
\end{equation}

Let us denote by $J_h$ the line generated by $h$, and set $J_{n,h}(x):=A_n(x) \cdot J_h - c_n(x)$.
\begin{lem}
For any $n_1\geq n_0$, there exists $n \geq n_1$ such that $J_{n,h}(x)$ does not pass through $0$, i.e., $J_{n,h}(x)\in \mathcal{J}$.
\end{lem}
\begin{proof}
Fix $n_1 \geq n_0$. Assume by contradiction that $J_{n,h}(x)$ passes through $0$ for all $n \geq n_1$. Since from (\ref{3.2}) and (\ref{3.3}), the matrices $A_n(x)$ expand $\Theta=\overline{\mathbb{P}_+^{d-1}}$, and because $J_{n,h}(x) \in \Theta$ with $\|J_{n,h}(x)\|<\delta$, we obtain
\begin{equation}\label{aux2}
    \lim\limits_{n \to \infty}\left\|A_{n-n_1}(T^{n_1}(x))^{-1}(A_n(x) \cdot th - c_n(x))\right\| = 0.
\end{equation}
But 
\begin{align*}
&A_{n-n_1}(T^{n_1}(x))^{-1}(A_n(x) \cdot t h - c_n(x))\\ &= A_{n_1}(x) \cdot t h - A_{n-n_1}(T^{n_1}(x))^{-1} \cdot c_n(x)\\
&= \underbrace{A_{n_1}(x) \cdot t h - c_{n_1}(x)}_{\|\cdot\|<\delta,\ \mathrm{constant\ w.r.t.\ }n}+\underbrace{c_{n_1}(x)+A_{n-n_1}(T^{n_1}(x))^{-1}\cdot c_n(x)}_{\in\ \Z^d}
\end{align*} so from (\ref{aux2}), $c_{n_1}(x)+A_{n-n_1}(T^{n_1}(x))^{-1}\cdot c_n(x)=0$ and $A_{n_1}(x) \cdot t h - c_{n_1}(x)=0$, whence $t h=(A_{n_1}(x))^{-1}\cdot c_{n_1}(x) \in \Z^d$ (since $A_{n_1}(x) \in \mathrm{SL}(d,\Z)$), a contradiction.
\end{proof}

Let $x=[\lambda] \in \mathcal{E}_{t,h}$. By the previous lemma, there are infinitely many $n \geq n_0(x)$ such that $J_{n,h}(x)\in \mathcal{J}$. Let us choose such an $n$.
By (\ref{propth}), we also know that $J_{n,h}(x) \cap W_\delta^s(T^n(x)) \neq \emptyset$. From the definition introduced in (\ref{defgamma}), we get $T^n(x) \in \Gamma_\delta^m(J_{n,h}(x))$, for any $m \geq 0$, that is,
\begin{equation*}
x \in \bigcap\limits_{m\geq 0}
T^{-n}(\Gamma_\delta^m(J_{n,h}(x))).
\end{equation*}
Let us define $\mathcal{J}_h:=\{M \cdot J_h - c\ \vert\ (M,c) \in \mathrm{SL}(d,\Z) \times \Z^d\} \cap \mathcal{J}$. By definition, we have that for every $n \geq 0$, $J_{n,h}(x)$ belongs to the countable set $\mathcal{J}_h$. We have thus obtained
\begin{equation*}
\bigcup\limits_{t \in \R,\ th \not \in \Z^d}\mathcal{E}_{t,h}\subset \mathcal{E}_h:=\bigcup\limits_{n\geq 0,\ J \in \mathcal{J}_h} \bigcap\limits_{m\geq 0} 
T^{-n}(\Gamma_\delta^m(J)).
\end{equation*}
Moreover, we deduce from estimate (\ref{esthd2}) that
$$
\mathrm{HD}(\mathcal{E}_h) \leq d-1-\min(\kappa,\alpha_1)< d-1.
$$\qed

\subsection{Weak mixing for flows: proof of Theorem \ref{wmixt}}\label{proofweakmixtransl}

We consider the cocycle $(T,A)$ introduced in Subsection \ref{sect flows trans}. Recall that by Subsection \ref{rauzyfast}, we know that there exists a probability measure $\mu$ for which $T$ has the property of bounded distortion; we denote by $\{\Delta^{(l)}\}_{l \in \mathbb{Z}}$ the associate partition. Moreover, the cocycle $(T,A)$ is fast decaying.
Let $L>0$ be the maximal Lyapunov exponent, i.e., $L:=\lim_{n \to +\infty} \frac{1}{n} \int_\Delta \ln \|A_n(x)\| d\mu (x)$. In the next proposition, we estimate the measure of the set of points with anomalous Lyapunov behavior, that is, those for which the iterates of the cocycle $(T,A)$ grow too fast. 
This large deviations result follows from fast decay of the cocycle and almost independence, which itself is implied by bounded distortion.
\begin{prop}\label{prop anom lyap behh}
For any $\widetilde L>L$, there exist $\widetilde C,\xi>0$ such that for every $n \geq 0$, 
$$
\mu(\{x \in \Delta\ \vert\ \|A_n(x)\|\geq e^{\widetilde L n}\})\leq \widetilde C e^{-\xi n}.
$$
\end{prop}

\begin{proof}
For any $(x,n) \in \Delta \times \mathbb{N}$, we set $I(x,n) := \frac{1}{n} \ln\|A_n(x)\|$; similarly, for any word $\underline{l}\in \Omega$ of length $n$, we set $I(\underline{l}):=\frac{1}{n} \ln\|A^{\underline{l}}\|$. By definition of $L$ and $\widetilde L$, and by bounded distortion (see Subsection \ref{backstrong} for the definition), there exist $n_0 \geq 0$ and $\theta>0$  such that for any $\underline{l} \in \Omega$, any $n \geq n_0$, we have $\int_\Delta (I(x,n)-\widetilde L)d\mu^{\underline{l}}(x) \leq -\theta$. Given an integer $n \geq 0$ and a word $\underline{l}\in \Omega$, let us consider the function $\mathscr{F}_{n,\underline{l}}\colon t \mapsto \int_\Delta e^{tn (I(x,n)-\widetilde L)}d\mu^{\underline{l}}(x)$ (we just write $\mathscr{F}_{n}$ when $\underline{l}=\emptyset$); it is well defined by fast decay of $A$. Moreover, we have $\mathscr{F}_{n_0,\underline{l}}(0)=1$ and $\mathscr{F}'_{n_0,\underline{l}}(0) \leq -n_0\theta<0$. Therefore there exist $\delta_0,\xi_0>0$ such that for any $\underline{l}\in \Omega$, we have $\mathscr{F}_{n_0,\underline{l}}(\delta_0) \leq e^{-\xi_0 n_0}$ and for any $0\leq n< n_0$,  $\mathscr{F}_{n,\underline{l}}(\delta_0) \leq 2 e^{-\xi_0 n}$. By definition, for any two words $\underline{l},\underline{l}'$, we have $\mu(\Delta^{\underline{l}\underline{l}'}) =T_*^{\underline{l}}(\restriction{\mu}{\Delta^{\underline{l}}})(\Delta^{\underline{l}'}) = \mu (\Delta^{\underline{l}}) \mu^{\underline{l}} (\Delta^{\underline{l}'})$. Besides, for any $x \in\Delta$ and integers $n,n' \geq 0$, we have $(n+n')I(x,n+n')\leq n I(x,n) + n' I(T^n(x),n')$. Let $n\geq 0$ be any integer. We deduce that 
\begin{align*}
\mathscr{F}_{n+n_0}(\delta_0)&=\sum_{|\underline{l}|=n,|\underline{l}'|=n_0}  e^{\delta_0(n+n_0) (I(\underline{l} \underline{l}')-\widetilde L)}d\mu(\Delta^{\underline{l} \underline{l}'})\\
&\leq \sum_{|\underline{l}'|=n_0}  e^{\delta_0 n_0 (I( \underline{l}')-\widetilde L)}d\mu^{\underline{l}}(\Delta^{ \underline{l}'})\sum_{|\underline{l}|=n}  e^{\delta_0 n (I( \underline{l})-\widetilde L)}d\mu(\Delta^{ \underline{l}})\\
&\leq \sup_{|\underline{l}|=n}\mathscr{F}_{n_0,\underline{l}} (\delta_0)\cdot \mathscr{F}_{n} (\delta_0) \leq e^{-\xi_0 n_0} \mathscr{F}_{n} (\delta_0).
\end{align*}
Therefore, for any $n \geq 0$, considering the euclidean division $n=kn_0+r$ of $n$ by $n_0$, with $k\geq 0$ and $0 \leq r < n_0$, we obtain $\mathscr{F}_{n}(\delta) \leq  e^{-k \xi_0 n_0} \cdot 2 e^{-\xi_0 r} \leq 2 e^{-\xi_0 n}$. We conclude by noting that $\mu\{x \in \Delta\ \vert\ \|A_n(x)\|\geq e^{\widetilde L n}\} \leq \mathscr{F}_{n}(\delta_0)\leq 2 e^{-\xi_0 n}$.
\end{proof}

We now come to the proof of Theorem \ref{wmixt}. Fix $0<\delta < 1/10$ sufficiently small. With the notations introduced in Subsection \ref{sect flows trans}, let $(h,[\lambda]) \in \mathcal{F}$; we denote $J_h:=\mathrm{Span}(h)$, and $\mathcal{J}_h:=\{M \cdot J_h - c\ \vert\ (M,c) \in \mathrm{SL}(d,\Z) \times \Z^d\} \cap \mathcal{J}$. As previously, we have
\begin{equation}\label{eqlambdaa}
[\lambda] \in \bigcup\limits_{n\geq 0,\ J \in \mathcal{J}_h}  \bigcap\limits_{m \geq 0}T^{-n}\left( \Gamma_{\delta/2}^m(J)\right).
\end{equation}

Fix $n\geq 0$, and take $(M,c) \in \mathrm{SL}(d,\Z) \times \Z^d$ such that $J=J_{M,c}(h):=M \cdot J_{h} -c \in \mathcal{J}_h$. We know from Proposition \ref{mainresultprob} that there exist constants $C,\kappa,\rho>0$ such that for any $m \geq 0$, $\mu\left(T^{-n}\left( \Gamma_{\delta}^m(J)\right)\right) \leq C e^{-\kappa m} \|J\|^{-\rho}$. Recall that we denote by $L$ the maximal Lyapunov exponent of $(T,A)$. Given $L'>NL$, let $X_m:=\{x \in \Delta\ \vert\ \|A_{mN}(x)\|\geq e^{L' m}\}$; by Proposition \ref{prop anom lyap behh}, we  know that $\mu(T^{-n}(X_m)) \leq \widetilde C e^{-\xi {m}}$ for certain constants $\widetilde C,\xi>0$. Then the proof of Theorem \ref{thm:HD} given in \cite{AD}  provides us with constants $C',\delta',K >0$ together with a cover $\{B^{m,n, M,c,h}_{k}\}_{k \geq 0}$ of $T^{-n}\left( \Gamma_{\delta}^m(J)\cup X_{m}\right)$ by balls of
diameter at most $e^{-C' m}$, satisfying
\begin{equation}\label{HD1}
\sum_{k \geq 0} \mathrm{diam}(B^{m,n,M,c,h}_{k})^{d-1-\delta'} \leq K.
\end{equation}

Recall that $\Gamma^m_{\delta}(J):=\{x \in \Delta,\ J\cap W_{\delta,mN}^s(x) \neq \emptyset\}$ is the set of points $x \in \Delta$ for which the process introduced earlier lasts for at least $m$ steps. By definition, $x \in T^{-n}(\Gamma_{\delta}^m(J))$ if and only if there exists $w \in J \cap B_{\delta}(0)$ such that 
$$
\|A_k(T^n(x)) \cdot w\|_{\R^d/\Z^d} < \delta,\quad \forall k \leq mN.
$$

If instead of $J$, we start the process with a line $J'$ close to $J$, we see that the first iterates of $J' \cap B_{\delta/2}(0)$ under the cocycle remain close to those of $J \cap B_{\delta/2}(0)$. More precisely, given  $0<\varepsilon<\delta/2 \cdot e^{-L' m}$, assume that the line $J'$ is $\varepsilon-$close to $J$ in the sense that for any $w' \in J' \cap B_{\delta/2}(0)$, $\|\pi_{J}(w') - w'\| < \varepsilon$, where $\pi_{J}$ denotes the orthogonal projection on $J$. Let $x' \in T^{-n}(\Gamma_{\delta/2}^m(J'))$; then there exists $w' \in J' \cap B_{\delta/2}(0)$ such that $\|A_k(T^n(x')) \cdot w'\|_{\R^d/\Z^d} < \delta/2$, for any $k \leq mN$.
Set $w:=\pi_{J}(w') \in J\cap B_{\delta}(0)$. Since $k \mapsto \|A_k(T^n(x'))\|$ is increasing, we deduce that either $x' \in T^{-n}(X_m)$, or for any $0\leq k \leq mN$,
$$
\|A_k(T^n(x')) \cdot w\|_{\R^d/\Z^d} < \delta/2 + \|A_k(T^n(x'))\| \varepsilon<\delta/2(1+ \|A_k(T^n(x'))\| e^{-L' m})<\delta,
$$
hence $x'\in T^{-n}(\Gamma_{\delta}^m(J))$. 
Therefore, 
there exists $C''>0$ depending only on $\delta, L'$ such that for each $h' \in H(\pi)$ with $d(h,h') < e^{-C'' m}$, and if $J':=J_{M,c}(h')$ denotes the corresponding line, then the same cover will work for the ``bad'' parameters associated with $J'$, that is, 
$$T^{-n}\left( \Gamma_{\delta/2}^m(J')\right) \subset \bigcup_k B^{m,n,M,c,h}_{k},$$
and $[\lambda'] \in T^{-n}\left( \Gamma_{\delta/2}^m(J')\right)$ whenever $(h',[\lambda']) \in \mathcal{F}$.

Let us choose a cover $\{U_j\}_{j\geq 0}$ of $H(\pi)\sim \mathbb{R}^{2g}$ by open balls of diameter less than $1$.  For any $j,m \geq 0$, we take a countable collection $\{h_{m,p}^j\}_{p \geq 0} \subset U_j$ such that $U_j \subset \cup_p B(h_{m,p}^j,\delta_{m,p}^j)$, where $0<\delta_{m,p}^j\leq e^{-C'' m}$ for all $p \geq 0$, and such that for some $0<K'<+\infty$, we have
\begin{equation}\label{eq 2g HD}
\sum_{p \geq 0} (\delta_{m,p}^{j})^{2g} \leq K'.
\end{equation}
From the previous discussion, we may also assume that there exists a constant $0 < K <+\infty$ such that for any integers $j,m,p,n\geq 0$, 
any $M\in \mathrm{SL}(d,\mathbb{Z})$, $c\in \Z^d$, there exists
an open cover $\{B_k^{m,n,M,c,h_{m,p}^j}\}_{k\geq 0}$ of the set $T^{-n}(\Gamma_{\delta}^m(J_{M,c}(h_{m,p}^j))\cup X_m)$ by balls of diameter at most $e^{-C' m}$, which satisfies 
\begin{equation}\label{HD12}
\sum_{k\geq 0} \mathrm{diam}(B^{m,n,M,c,h_{m,p}^j}_{k})^{d-1-\delta'} \leq K,
\end{equation} 
and such that for any $(h,[\lambda]) \in \mathcal{F}$, we have: 
\begin{itemize}
\item $h \in U_j$ for some $j \geq 0$, and for any $m \geq 0$, there exists $p \geq 0$ such that $h \in B(h_{m,p}^j,\delta_{m,p}^j)$.
\item There exist $n\geq 0$, $M\in \mathrm{SL}(d,\mathbb{Z})$, $c\in \Z^d$ such that 
 $J_{M,c}(h) \in \mathcal{J}_h$ and for any $m\geq 0$,  $[\lambda] \in T^{-n}(\Gamma_{\delta/2}^m(J_{M,c}(h)))\subset \cup_k B_k^{m,n,M,c,h_{m,p}^j}$.
\end{itemize} 
For each $j,m,n,M,c,p,k$ as above, set $r_{k}^{j,m,n,M,c,p}:=\mathrm{diam}(B^{m,n,M,c,h_{m,p}^j}_{k})$, and take a cover $\{\mathcal{B}_l^{j,m,n,M,c,p,k}\}_{l}$ of $B(h_{m,p}^j,\delta_{m,p}^j)$ by $O\left(\left(\frac{\delta_{m,p}^j}{r_{k}^{j,m,n,M,c,p}}\right)^{2g}\right)$ many balls of diameter $r_{k}^{j,m,n,M,c,p}$. We get
\begin{equation}\label{Funion2}
\mathcal{F} \subset \bigcup_{j,n,M,c} \bigcap_{m \geq 0} \left(\bigcup_{p,k,l} \mathcal{B}_l^{j,m,n,M,c,p,k} \times B^{m,n,M,c,h_{m,p}^j}_{k} \right).
\end{equation}
For any $j,n,M,c$, and for every integer $m \geq 0$, let us denote $Y_{m}^{j,n,M,c}:=\bigcup_{p,k,l}\mathcal{B}_l^{j,m,n,M,c,p,k} \times B^{m,n,M,c,h_{m,p}^j}_{k} $. We deduce from \eqref{eq 2g HD}  and \eqref{HD12} that
\begin{align*}
\sum_{p,k,l} \mathrm{diam}(\mathcal{B}_l^{j,m,n,M,c,p,k} \times B^{m,n,M,c,h_{m,p}^j}_{k})^{2g+d-1-\delta'} &\lesssim \sum_{p,k} \left(\frac{\delta_{m,p}^j}{r_{k}^{j,m,n,M,c,p}}\right)^{2g} (r_k^{j,m,n,M,c,p})^{2g+d-1-\delta'}\\ 
&= \sum_{p} (\delta_{m,p}^j)^{2g} \times \sum_k (r_k^{j,m,n,M,c,p})^{d-1-\delta'}\\ 
&\leq K K'.
\end{align*}
We deduce that $\mathrm{HD}\left(\bigcap_{m \geq 0} Y_m^{j,n,M,c}\right) \leq 2g + d-1-\delta'$. But from \eqref{Funion2}, we know that $\mathcal{F}$ is contained in a countable union of such sets, so that $\mathrm{HD}(\mathcal{F}) < 2g+d-1$, which ends the proof.\qed


\begin{thebibliography}{BKNS}

\bibitem[AD]{AD}  {Avila} A.,  {Delecroix} V.; \textit{Weak mixing directions in non-arithmetic Veech surfaces}, Journal of the American Mathematical Society \textbf{29} (2016), pp. 1167--1208.

\bibitem[AF]{AF1}  {Avila} A.,  {Forni} G.; \textit{Weak mixing for interval exchange transformations and translation flows}, Ann. of Math. \textbf{165} (2007), pp. 637--664.

\bibitem[AGY]{AGY} {Avila} A., {Gou\"{e}zel} S., {Yoccoz} J.-C.; \textit{Exponential mixing for the Teichm\"{u}ller flow}, Publications math\'{e}matiques de l'IH\'{E}S \textbf{104}, 143-211.

\bibitem[F1]{F1} {Forni} G.; \textit{Solutions of the cohomological equation for
area-preserving flows on higher genus surfaces}, Ann.
       of Math. (2) \textbf{146} (1997),  295--344.

\bibitem[F2]{F2} {Forni} G.; \textit{Deviation of ergodic
       averages for area-preserving flows on surfaces of higher genus}, Ann.
       of Math. (2) \textbf{155} (2002), no. 1, 1--103.

\bibitem[GK]{GK} {Gutkin} E., {Katok} A. B.;  \textit{Weakly mixing billiards},  {\it Holomorphic dynamics (Mexico,
1986)}, 163--176. Lecture Notes in Math., \textbf{1345}, Springer, Berlin, 1988.

\bibitem[Ka]{Ka} {Katok} A. B.; \textit{Interval exchange transformations and some special flows
are not mixing}, Israel J. Math. \textbf{35} (1980) no. 4, 301--310.

\bibitem[KS]{KS1} {Katok} A. B., {Stepin} A. M.;
\textit{Approximations in ergodic theory}, Uspehi Mat. Nauk \textbf{22} 1967 no. 5 (137), 81--106.

\bibitem[Ke]{Ke} {Keane} M. S.; \textit{Interval exchange transformations}, Math. Z. \textbf{141} (1975), 25--31.

\bibitem[KMS]{KMS} {Kerckhoff} S., {Masur} H., {Smillie} J.; \textit{Ergodicity of billiard flows and quadratic
differentials},  Ann. of Math. (2) \textbf{124} (1986), 293--311.

\bibitem[Ko]{Ko} {Kontsevich} M.; \textit{Lyapunov exponents and Hodge theory}, in  \textit{The mathematical
beauty of physics} (5-7 juin 1996, Saclay), J.M. Drouffe and J.B. Zuber (Eds.),
318--332, Adv. Ser.  Math. Phys. \textbf{24}, World Sci. Publishing, River Edge, NJ, 1997.

\bibitem[L]{L} {Lucien} I.;
\textit{M\'{e}lange faible topologique des flots sur les surfaces},
Ergodic Theory Dynam. Systems \textbf{18} (1998), no. 4, 963--984.

\bibitem[MMY]{MMY} {Marmi} S., {Moussa} P., {Yoccoz} J.-C.; \textit{The cohomological equation for Roth type interval exchange maps}, Jour. Amer. Math. Soc. \textbf{18} (2005), 823-872.

\bibitem[Ma1]{Ma1} {Masur} H.;
   \textit{Interval exchange transformations and measured foliations}, Ann. of Math. (2) \textbf{115} (1982), no. 1, 169--200.

\bibitem[Ma2]{Ma2} {Masur} H.; \textit{Hausdorff dimension of the set of nonergodic foliations of a quadratic differential}, Duke Math. J. \textbf{66} (1992), no. 3, 387--442.

\bibitem[NR]{NR} {Nogueira} A., {Rudolph} D.;
\textit{Topological weak-mixing of interval exchange maps},
Erg. Th. Dyn. Syst. \textbf{17} (1997), no. 5, 1183--1209.

\bibitem [R]{R} {Rauzy} G.;
 \textit{\'Echanges d'intervalles et transformations induites}, Acta Arith. \textbf{34} (1979), 315--328.

\bibitem[V1]{V1} {Veech} W. A.;
\textit{Projective swiss cheeses and uniquely ergodic interval exchange transformations},
{\it Erg. Th. Dyn. Syst. I, Proc. Special Year, Maryland 1979-1980},
 A. Katok ed., 113--195, Birkha\"user, 1981.

\bibitem[V2]{V2} {Veech} W. A.;
\textit{Gauss measures for transformations on the space of interval exchange
maps}, Ann. of Math. (2) \textbf{115} (1982), no. 1, 201--242.

\bibitem[V3]{V3} {Veech} W. A.; \textit{The Teichm\"uller geodesic
       flow}, Ann. of Math. (2) \textbf{124} (1986), no. 3, 441--530.

\bibitem[V4]{V4} {Veech} W. A.; \textit{The metric theory of
       interval exchange transformations. I. Generic spectral properties},
       Amer. J. Math. \textbf{106} (1984), no. 6, 1331--1359.

\bibitem[V5]{V6} {Veech} W. A.; \textit{Moduli spaces of quadratic differentials},
       Journal d'Analyse Math. \textbf{55} (1990), 117--171.

\bibitem[Via]{Via} {Viana} M.; \textit{Ergodic Theory of Interval Exchange Maps}, Revista Matem\'{a}tica Complutense \textbf{19} Num. 1, 2006.

\bibitem[Z1]{Z1} {Zorich} A.;
\textit{How do the leaves of a closed $1$-form wind around a surface?}, Pseudoperiodic Topology,
135-178, Amer. Math. Soc. Transl. Ser. 2, \textbf{197}, Amer. Math. Soc., Providence, RI, 1999.

\bibitem[Z2]{Z2} {Zorich} A.;
\textit{Flat surfaces}, in collection ``Frontiers in Number Theory, Physics and Geometry, Vol. 1 : On random matrices, zeta functions and dynamical systems", Springer Verlag, Berlin, 2006, 439--586.

\end{thebibliography}
\end{document}